\documentclass{amsart}
\usepackage[a4paper, right=3.5 cm, left=3.5cm, top=5cm, bottom=5cm]{geometry}
\usepackage{graphicx} 
\usepackage{amsmath} 
\usepackage{amsfonts} 
\usepackage{amssymb} 
\usepackage{hyperref} 
\usepackage{cite} 
\usepackage{lipsum} 
\usepackage{bbold}
\usepackage{amsthm}
\usepackage{float}  
\usepackage[dvipsnames]{xcolor}
\usepackage[toc,page]{appendix}
 
\newtheorem{theorem}{Theorem}[section]
\newtheorem{proposition}[theorem]{Proposition}

\newtheoremstyle{example}
  {.3\baselineskip}
  {.3\baselineskip}
  {\normalsize}  
  {0pt}       
  {\bfseries} 
  {.}         
  {5pt plus 1pt minus 1pt} 
  {}          
\theoremstyle{example}
\newtheorem{example}[theorem]{Example}

 \usepackage{fancyhdr}
 \usepackage{chngpage}
\newtheorem{lemma}{Lemma} 
\newtheorem{remark}{Remark}

\usepackage[english]{babel}

\title[Energy derivatives sensitivities]{EFFICIENT computation of SENSITIVITIES FOR derivatives IN ENERGY MARKETS}

\author[Benth]{Fred Espen Benth} 
\address{
Department of Data Science and Analytics \\
BI Norwegian Business School \\
N-0484 Oslo, Norway \& 
    Department of Mathematics, University of Oslo \\ 
    PO Box 1053 Blindern, N-0316 Oslo, Norway \\ 
   E-mail: \href{mailto:fred.e.benth@bi.no}{fred.e.benth@bi.no} }
\author[Draouil]{Olfa Draouil}
\address{
    Department of Mathematics, Stochastic Analysis and Applications Research Laboratory\\ 
    Faculty of Sciences of Tunis, University of Tunis El Manar  
    Tunis, Tunisia\\
      E-mail:\href{mailto:olfa.draouil@fst.utm.tn}{olfa.draouil@fst.utm.tn} }
\author[Hammami]{Farouk Hammami}
\address{Department of Mathematics, Stochastic Analysis and Applications Research Laboratory  \\
    Faculty of Sciences of Tunis, University of Tunis El Manar 
    Tunis, Tunisia\\
     E-mail:\href{mailto:farouk.hammami@fst.utm.tn}{farouk.hammami@fst.utm.tn} }

\date{\today }

\begin{document}
  

\begin{abstract}
In this study, we develop a stochastic framework for computing Delta sensitivities in energy markets, where both prices and traded volumes are modeled as correlated stochastic processes. Within this framework, we analyze two complementary approaches for sensitivity analysis: the {\it density method}, which is applicable when the density of the underlying process is known, and the {\it Malliavin calculus method}, which does not require any explicit knowledge of the density and relies only on the dynamics of the processes.   
We present illustrative examples for both methods. For the density-based approach, we consider Ornstein–Uhlenbeck and CARMA processes to model prices and energy volumes. For the Malliavin calculus approach, we study Ornstein–Uhlenbeck processes, jump diffusion driven by a compound Poisson process, time-changed Brownian motion processes subordinated  by an inverse Gaussian (IG) process, as well as Ornstein–Uhlenbeck processes driven by a normal inverse Gaussian (NIG) process. We provide some numerical examples illustrating the implementation of the proposed formulas and demonstrating a close agreement between the resulting delta estimates.
\end{abstract}
\maketitle


\section{ Introduction } 
The latest data indicate that the market for weather derivatives has continued to witness renewed activity, especially because of the volatility of the climatic factors. The Chicago Mercantile Exchange (CME) recorded a major boost in trade volumes during 2023, registering more than 260\% growth in listed weather derivative contracts \cite{CME2024}. More recent figures indicate that the number of outstanding weather derivative
contracts was approximately 48\% higher in May 2024 than in May 2023,
reflecting the continued growth of the market \cite{CME2024}. However, despite the recent increase in activity, the total volume of weather derivatives is still quite small compared to other futures contracts, such as crude oil. Customized over-the-counter (OTC) products, such as ``quanto'' weather contracts, continue to lead the market, with a notional value of around \$25 billion in 2024 \cite{Governing2024}.
Quanto options are options based on both the price and a volume-related index. The volume aspect of the contract is usually measured by temperature to reflect the power demand. As renewable power grows, this volume could also be measured using irradiation (solar power) or wind speed (wind power), or even both. The interest in quanto options is growing because companies are becoming more aware of weather-related risks, especially in industries like energy and agriculture, which are very sensitive to weather changes.
The term "quanto options" typically refers to derivatives in financial markets that allow investors to gain exposure to price changes in foreign assets without facing exchange rate risks.

Caporin and McAleer \cite{cap} introduced a bivariate time-series model to capture the relationship between energy prices and temperature. Their model seeks to address seasonality, long memory, autoregression, and dynamic correlations between the two factors. Because of this complexity, their approach uses simulation methods for pricing, and the challenge of effectively hedging these options remains open. For quanto contracts to be useful as risk management tools, they show a strong correlation between their underlying assets. In energy markets, the payoff from a quanto option depends on both the energy price and a weather-related index. Research by Engle et al. \cite{Engleetal1992} highlights the importance of temperature in predicting electricity prices, while Timmer and Lamb \cite{TimmerLamb2007} establish a strong connection between natural gas prices and heating degree days (HDD).
 In \cite{benth2010computation}, the authors derive explicit pricing formulas and Greeks for standard (non-quanto) energy derivatives under multifactor models, providing closed-form expressions that are widely used in hedging applications.
In contrast, our paper models both the price and the energy assets using a multifactor dynamic framework. We apply the 
density method to compute the deltas, first with respect to the price and then with respect to the energy. This approach is close to the one used in \cite{benth2010computation}, although that study only considers a single underlying asset (the price) and does not focus on quanto options. 
Furthermore, we employ a Malliavin calculus approach to compute the deltas. While a similar method is explored in \cite{benth2010robustness}, it does not address quanto options.
Our findings are illustrated through a series of examples.
In our approach, we use a time-changed Brownian motion to model both price and volume. To the best of our knowledge, this is the first model to simultaneously utilize time-change dynamics for both variables in the context of quanto option pricing.

From a methodological perspective, it is important to emphasize that both the density method and Malliavin approaches operate within essentially the same class of financial models, but with different technical requirements that necessitate distinct formulations. The density method relies fundamentally on the existence of joint probability densities for the underlying factors, which permits computations at the level of distribution functions without explicit specification of the stochastic dynamics. This approach offers significant flexibility in exploiting the structural properties of the model while maintaining analytical tractability. Conversely, the Malliavin calculus approach necessitates a more concrete specification of the factor dynamics through systems of stochastic differential equations, as it operates directly on the level of the driving noise processes—typically Brownian motions or Lévy processes. This requirement for explicit dynamics is counterbalanced by the dimensional flexibility of the methodology. Both frameworks ultimately provide complementary pathways to sensitivity analysis.

\subsection{ Some mathematical preliminaries}  
Let  $ \left ( \Omega, \mathcal{F} ,(\mathcal{F}_t)_{0\leq t\leq T},\mathbb{Q}\right )$ be a filtered probability space  where $T$ is some finite time horizon. Suppose that the energy spot  price   $S(t), 0\leq t\leq T $ is a stochastic process given by
\begin{equation}
    S(t) = h_{1}(t,X_1(t)+\zeta_1(S(0)) ,X_{2}(t) ,...,X_{n}(t) ),
\end{equation}
where $X_{1}(t) ,X_{2}(t) ,...,X_{n}(t)$ are $n$   adapted stochastic processes, $ \zeta _{1}$  is a differentiable function  and the function  $ h_{1}:   \mathbb{R}^{n+1} \rightarrow\mathbb{R}$  is Borel measurable ensuring that $S(t)$ is adapted.
We assume that $V(t), 0\leq t\leq T$ is another stochastic process representing, say, the solar irradiation, temperature, wind speed, or a combination of these. It is assumed to have the form given by 
\begin{equation}
   V(t) = h_{2}(t,Y_1(t)+\zeta_2(V(0)),Y_{2}(t) ,...,Y_{m}(t) ),
\end{equation}  
where  $Y_{1}(t) ,Y_{2}(t),...,Y_{m}(t)$ are $m$   adapted stochastic processes, $ \zeta _{2}$  is a differentiable function and the function  $ h_{2}:   \mathbb{R}^{m+1}  \rightarrow\mathbb{R}$  is Borel measurable.   
Notice that $x=h_1(0,X_1(0)+\zeta_1(x),X_2(0),...,X_n(0))$ where 
$X_1(0)=x , X_2(0)=X_3(0)=...=X_n(0)=0$ and $v=h_2(0,Y_1(0)+\zeta_2(v),Y_2(0),...,Y_m(0))$ where $Y_1(0)=v , Y_2(0)=Y_3(0)=...=Y_m(0)=0$.  As the value of $V$ typically is a determinant of the volume of energy produced (solar irradiation or wind speed), or the volume of power demanded (temperature), we shall refer to it as the {\it volume} process.  Notice that one or more of the factors in the $S$ and $V$-dynamics may be the same. For example, we can have $X_2=Y_2$ or some linear or even functional relationship. Furthermore, we may have that the factors depend on the same noise driver, being a Brownian motion or a L\'evy process. For example, if $V$ is high, one usually see low prices in the market. We can account for this by having the factors in $V$ and $S$ dependent. 

We consider European options written on the energy price $S$ and volume $V$,  yielding a Borel measurable payoff $ h: \mathbb{R} \times \mathbb{R}  \rightarrow\mathbb{R}$  at the exercise time $T$. Our aim is to price such derivatives, which from the arbitrage pricing theory is known to be given as the discounted expected value of the payoff under a pricing measure. Here and throughout the paper, we assume that we have modelled the processes $S$ and $V$ directly under the pricing measure, which is indicated by denoting the probability by $\mathbb Q$ in our probability space.   Assume that $\mathbb{E}  \left[ |h(S(T),V(T) )|^2\right]<\infty$. 
  In this paper, we focus our attention on
 the arbitrage free price at time zero, denoted by $C(0)$. As the models we shall use for the price and volume allow for expressing the option price in terms of the current spot $S(0)$ and volume $V(0)$, we shall resort to the notation  $C(S(0),V(0))$ since later  we will compute the Deltas with respect to $S(0)$ and $V(0)$: 
\begin{align}\label{arbi}
     C(0):=C(S(0),V(0))&=e^{-rT}\mathbb{E}\Big[h\Big(S(T),V(T)\Big)\Big]. 
     \end{align}
Here, $r>0$ is the rate risk free constant interest rate.
 We note in passing that $C(0)$ also depends on the initial conditions of the factors constituting the dynamics of $S$ and $V$. 
 For later purposes, it is convenient to have available also the following notation, re-expressing our payoff function $h$ explicitly in all its components:
\begin{equation}
    h(s,v) =G(T,x_{1}+\zeta_{1}(s_0),x_{2},....,x_{n},v_{1}+\zeta_{2}(v_0),y_{2},....,y_{m}).
\end{equation}
 
To make expressions more compact,we introduce a short-hand notation 
$$
X_{i:n}(t):=(X_i(t),...,X_n(t))
$$
for all $\; i, j \in \mathbb{N} \;(i < n)$, and similarly for the $Y$'s, 
$$
{Y}_{i:m}(t):=(Y_i(t),..., Y_m(t))
$$
for all $\; i, m \in \mathbb{N} \;(i < m)$.
The integrability condition on $h(S(T),V(T))$ can now be re-stated as:
 \begin{equation}
\label{G-condition}
    \mathbb{E} \Big[ |G(T,X_{1}(T)+\zeta_{1}(S(0)),{X}_{2:n}(T),Y_{1}(T)+\zeta_{2}(V(0)),{Y}_{2:m}(T))|^2\Big]<\infty.
\end{equation}
We next include an example to showcase our modeling framework and notation:
\begin{example}
A quanto option in the energy markets context may have a payoff function being the product of a call on price and put on volume (see e.g. \cite{benth2015pricing}),
\begin{equation}\label{callex}
h(S(T), V(T)) \;=\; \max \bigl(S(T) - K,\, 0\bigr) \;\times\; \max \bigl(E - V(T),\, 0\bigr).
\end{equation}
Here, $K$  is the strike price and $E$ is the  volume threshold. This option price provides protection against simultaneous high prices and low production volume. Now, the payoff of this quanto option can be re-phrased in our notation with the functions $h_1$ and $h_2$ by   
   \begin{align*}
      h(S(T),V(T))&=\max\big( h_{1}(T,X_1(t)+\zeta_1(S(0)),{X}_{2:n}(T))-K,0 \big)  \\
     & \qquad\times \max\big( E- h_{2}(T,Y_1(t)+\zeta_2(S(0)),{Y}_{2:m}(T) ),0 \big). 
   \end{align*}
    The functions $h_1$ and $h_2$ are linking $S$ and $V$ to the factors $X_{1}$,...,$X_{n}$ and $Y_{1}$,...,$Y_{m}$. In energy markets, one typically model prices and volume by multi-factor Ornstein-Uhlenbeck processes (see e.g. \cite{benth2024stochastic}). For example, a two-factor Schwartz model  has the form 
    \begin{equation}
        S(T)=S(0)e^{X_{1}(T)+X_{2}(T)} \quad S(0)>0,
    \end{equation}
    where $X_1$ and $X_2$ are two Ornstein-Uhlenbeck processes  satisfying the following SDEs
      \[
  \begin{aligned}
    dX_i(t) &= -\alpha_{X_i} X_i(t) dt + \sigma_{X_i} dW_i(t), 
  \end{aligned}
  \]
   for $i=1,2$, with $W_1, W_2$ being two independent Brownian motions and both factor processes are starting at zero. 
    In our notation, we find 
    $$
    h_{1}(T,x_{1} ,x_{2})=S(0)\exp(x_{1}+x_{2})=\exp(x_1+\ln(S(0))+x_2).
    $$
For more details on the Schwartz model, see \cite{pricesevidence}.
    Therefore, when modelling $S$, we set $\zeta_1(z)=\ln(z)$ and $h_1$ to be the exponential function. 
This choice of modeling guarantees that the price is non negative.
If we assume $V$ to be the wind speed, say, one may use a one-factor exponential Ornstein-Uhlenbeck process driven by a non-decreasing, pure-jump L\'evy process $L$.  (see \cite{benth2018non}),
\begin{equation*}
   V(T)=V(0)e^{Y_{1}(T)}, \quad V(0)>0,
\end{equation*}
where $Y_1$ is starting at zero. Hence, 
$$
h_{2}(T,y_{1}) =V(0)\exp(y_{1})=\exp(y_1+\ln V(0)).
$$
Again, we use an exponential function for $h_2$ and $\zeta_2(z)=\ln(z)$. 
Putting all this together, we find the function $G$ as
\begin{align*}
        G(T,x_1+\zeta_{1}(S(0)), x_2,y_1+\zeta_{2}(V(0)))
=\max(e^{x_1+\ln(S(0))+x_{2}}-K,0)\\
\qquad\times \max(E-e^{y_1+\ln(V(0))},0),
\end{align*}
  where   $\zeta_{1}(S(0))=\ln(S(0))$ and $\zeta_{2}(V(0))=\ln(V(0))$.
Hence, we can express the quanto option price as a function of the initial spot price and volume by
\[
   C(S(0),V(0)) =e^{-rT} \mathbb{E} \Big[ G(T,X_{1}(T)+\zeta_{1}(S(0)),X_{2}(T),Y_{1}(T)+\zeta_{2}(V(0))\Big].
\]
The process \(Y_1\) is assumed to be an Ornstein-Uhlenbeck process driven by a subordinator \(L\), where \(L\) is a non-decreasing pure-jump Lévy process. Moreover, we assume that \(L\) is independent of the Brownian motions \(W_1\) and \(W_2\) driving \(X_1\) and \(X_2\). Our model set-up thus assumes independence between volume and prices. However, one can remedy this by assuming for example that the factor $X_2$ in the price process is driven by $L$. To account for the so-called cannibalisation effect, where prices tends to drop when wind power generation increases, we could let the driving noise of $X_2$ be $-L$ rather than $W_2$, say, or we could add a third factor $X_3$ which depends linearly on $V$.   This ends our example.
\end{example}
Note that if $G$ is of at most of linear growth, then we need that each factor has a finite  square expected value in order for \eqref{G-condition} to be satisfied: 
 $\mathbb{E} \Big[|X_i|^2\Big]<\infty$ and $\mathbb{E} \Big[|Y_i|^2\Big]<\infty$  for $i=1,2,...$.
On the other hand, if $G$ is of exponential growth, as in the above example, then we need that the exponential moments of the factors are integrable:  $\mathbb{E} \Big[e^{X_i}\Big]<\infty$ and $\mathbb{E} \Big[e^{Y_j}\Big]<\infty$  for $i=1,..,n$ and $j=1,...,m$.

\section{Deltas with the density method}
  Glasserman \cite{Glasserman2003} proposed a so-called density method to compute Greeks in the case when the probability density function of the underlying processes is known. The idea is to move the parameter in question to the density function and differentiate this after re-expressing the expected value in terms of the density function. In this way, one avoids differentiating the payoff function. We use this approach in the context of energy derivatives in this Section. 
\subsection{The Delta with respect to  $S(0)$ and $V(0)$}
In Proposition \ref{prop 2.1} below,  we  assume that we have a dependency between price and volume factors. Let \( g : \mathbb{R}^{n+m} \to \mathbb{R}_+ \quad \) be the joint probability density function of \(\bigl(X_{1:n}(T),\, Y_{1:m}(T)\bigr).\) For convenience, we recall the notation:  
\[
x_{i:n} := (x_i, \dots, x_n), 
\quad 
y_{j:m} := (y_j, \dots, y_m).
\quad 
For\; all \; i, j \in \mathbb{N} \;(i < n)\; and \;(j < m) .
\]
and
\[
 dx_1 dx_2...dx_n=dx_1: dx_n\, \text{ and } dy_1 dy_2...dy_m=dy_1 :dy_m. 
\]
\begin{proposition}\label{prop 2.1}
\begin{enumerate}
\item[I)] Assume that the following hold:
\begin{itemize}
    \item The function 
    \(
    z \mapsto g\big(x_1 - z,\, x_{2:n},\, y_{1:m}\big)
    \)
    is differentiable almost everywhere.
    \item 
    \(
    \mathbb{E} \Big[ 
        \big| G\big(T,\, X_1(T) + z + \zeta_{1}(S(0)),\, X_{2:n}(T),\, Y_1(T) + \zeta_{2}(V(0)),\, Y_{2:m}(T)\big) \big|^2 
    \Big]^{1/2} < \infty.
    \)
       \item 
    There exists an integrable function \(\phi \in L^1(\mathbb{R}^{n+m})\) such that for all \(z\) in a neighbourhood of interest,
    \[
    \Big| 
    G\big(T,x_1 + \zeta_{1}(S(0)),\, x_{2:n},\,  y_1 + \zeta_{2}(V(0)),\, y_{2:m}\big)\,
    \frac{\partial}{\partial x_1} g(x_1 - z,\, x_{2:n},\, y_{1:m}) 
    \Big|
    \leq \phi(x_{1:n},\, y_{1:m}).
    \]
\end{itemize}

Then the delta with respect to the price is given by 
\begin{align*}
\Delta_S 
&= \frac{\partial C}{\partial S(0)} \\
&= -\, \zeta_1'(S(0))e^{-rT}\,
\mathbb{E} \Big[ 
   G\big(T,X_1(T)+ \zeta_1(S(0)),\, X_{2:n}(T),\,  Y_1(T) + \zeta_{2}(V(0)),\, Y_{2:m}(T)\big) \\
&\qquad \times 
   \frac{\partial}{\partial x_1} 
   \ln g\big( X_{1:n}(T),\,  Y_{1:m}(T)\big) 
\Big].
\end{align*}

\item[II)] Assume that the following hold:
\begin{itemize}
    \item The function 
    \(
    z \mapsto g\big( x_{1:n},\, y_1 - z ,\, y_{2:m}\big)
    \)
    is differentiable almost everywhere.
    
    \item 
    \( 
    \mathbb{E} \Big[ 
        \big| G\big(T,\, X_1(T) + \zeta_{1}(S(0)),\, X_{2:n}(T),\, Y_1(T) + z + \zeta_{2}(V(0)),\, Y_{2:m}(T)\big) \big|^2 
    \Big]^{1/2} < \infty.
    \)
    
    \item There exists an integrable function \(\phi \in L^1(\mathbb{R}^{n+m})\) such that for all \(z\) in a neighbourhood of interest,
    \[
    \Big| 
    G\big(T,x_1 + \zeta_{1}(S(0)),\, x_{2:n},\, y_1 + \zeta_{2}(V(0)),\, y_{2:m}\big)\,
    \frac{\partial}{\partial y_1} g(x_{1:n},\, y_1 - z,\, y_{2:m}) 
    \Big|
    \leq \phi(x_{1:n},\, y_{1:m}).
    \]
\end{itemize}

Then the delta with respect to the volume is given by 
\begin{align*}
\Delta_V 
&= \frac{\partial C}{\partial V(0)} \\
&= -\, \zeta_2'(V(0))e^{-rT}\,
\mathbb{E} \Big[ 
   G\big( T,X_{1:n}(T),\, Y_1(T)+ \zeta_2(V(0)),\, Y_{2:m}(T)\big) \\
&\qquad \times 
   \frac{\partial}{\partial y_1} 
   \ln g\big( X_{1:n}(T),\, Y_{1:m}(T)\big) 
\Big].
\end{align*}
\end{enumerate}
\end{proposition}

\begin{proof}

We define
\(
F(z) = e^{-rT}\mathbb{E}\left[G(T,X_1(T) + z, X_{2:n}(T), Y_1(T) + \zeta_{2}\big(V(0)), Y_{2:m}(T))\right].
\)
By a change of variable, we get 
\begin{align*}
F(z) &=e^{-rT}\int_{\mathbb{R}^{n+m}} G(T, x_1 + z, x_{2:n} , y_1 + \zeta_{2}\big(V(0)), y_{2:m} )\, g ( x_{1:n} , y_{1:m} )\, dx_1 : dx_n\, dy_1 : dy_m,\\
&=e^{-rT} \int_{\mathbb{R}^{n+m}} G(T,x_{1:n}, y_1 + \zeta_{2}\big(V(0)), y_{2:m} )\, g (x_1 - z, x_{2:n} ,  y_{1:m} )\, dx_1:dx_n\, dy_1 : dy_m.
\end{align*}
Thanks to the assumptions of I) of Proposition \ref{prop 2.1}, we apply the  dominated convergence theorem to derive  
\begin{align*}
F'(z)& =e^{-rT}\frac{\partial}{\partial z} \int_{\mathbb{R}^{n+m}} G(T,x_1, x_{2:n}, y_1+ \zeta_{2}\big(V(0)), y_{2:m} )\, g (x_1 - z, x_{2:n} ,  y_{1:m} )\, \\&\qquad \times dx_1  : dx_n\, dy_1 : dy_m\\
&= -e^{-rT} \int_{\mathbb{R}^{n+m}} G(T,x_1, x_{2:n}, y_1  + \zeta_{2}\big(V(0)), y_{2:m}  )\, \frac{\partial}{\partial x_1} g (x_1 - z, x_{2:n} , y_1, y_{2:m}  )\, \\&\qquad \times dx_1 : dx_n\, dy_1 :dy_m.
\end{align*}
After multiplying and dividing by  $   g (x_1 - z, x_{2:n} , y_1, y_{2:m})  $, we have   
\begin{align*}
&\frac{\partial}{\partial x_1}g(x_1 - z, x_{2:n} , y_1, y_{2:m} )\\
&= \frac{ \frac{\partial}{\partial x_1}g(x_1 -  z, x_{2:n} , y_1, y_{2:m} ) }{ g (x_1 - z, x_{2:n} , y_1, y_{2:m}) } 
\times g (x_1 - z, x_{2:n} , y_1, y_{2:m} ) \\
&= \frac{\partial}{\partial x_1} \ln g (x_1 - z, x_{2:n} , y_1, y_{2:m}) 
\times g (x_1 - z, x_{2:n} , y_1, y_{2:m} ).
\end{align*}
Then,
\begin{align*}
F'(z)
&= - e^{-rT}\,
\mathbb{E}\Bigg[
G\Big(
T,\,
X_1(T)+z,\,
X_{2:n}(T),\,
Y_1(T)+\zeta_{2}\big(V(0)\big),\,
Y_{2:m}(T)
\Big) \\
&\hspace{2cm}\times
\frac{\partial}{\partial x_1}
\ln g\big(X_{1:n}(T),\, Y_{1:m}(T)\big)
\Bigg].
\end{align*} 
Therefore,
\begin{align*}
\Delta_S 
&= \frac{\partial C}{\partial S(0)}\nonumber \\
&= \frac{\partial}{\partial S(0)} F(\zeta_1(S(0)))\nonumber \\
&= \zeta_1'(S(0)) \, F'(\zeta_1(S(0))) \nonumber\\
&=- \zeta_1'(S(0))e^{-rT}\mathbb{E} \Big[ G\big(T,X_1(T) + \zeta_1(S(0)),\, X_{2:n}(T),\, Y_1(T)+ \zeta_{2}\big(V(0)),\, Y_{2:m}(T)\big)\nonumber \\
&\quad \times \frac{\partial}{\partial x_1} \ln g\big( X_{1:n}(T),\,  Y_{1:m}(T)\big) \Big].\nonumber \\
\end{align*}
By analogy to the proof of the delta with respect to the price, we prove the delta with respect to the volume after appealing to the assumptions II) in the Proposition. 
\end{proof}

\begin{remark}
     Note that Proposition \ref{prop 2.1} is related to Proposition 3.1 in \cite{benth2010computation}. The main difference is that in our context the payoff function depends not only on the price but also on the volume represented by $V(T)$. In addition, note that our Proposition \ref{prop 2.1} is more general since the factors are dependent, while in Proposition 3.1 in \cite{benth2010computation} the first factor $X_1$ is independent from the others factors.
 \end{remark}

\section {Applications of the density method}

In this section we will  apply the density method to compute the Delta for various stochastic models, including one- and two-factor price models, a one-factor wind or temperature model, and a CARMA-type model. 

\subsection{ A  one-factor volume and price model with correlation  }\label{example 1  1P 1V}
 Consider the following system of stochastic differential equations:
  \[
  \begin{aligned}
    dX(t) &= -\alpha_X X(t) dt + \beta_X dW(t), \\
    dY (t) &= -\alpha_Y Y(t) dt + \eta d\tilde{W}(t),
  \end{aligned}
  \]
  where \( X(t) \) and \( Y(t) \) are Ornstein-Uhlenbeck processes and
  \( W(t) \) and \( \tilde{W}(t) \) are correlated Brownian motions expressed through the relation
\[
dW(t) = \rho \, d\tilde{W}(t) + \sqrt{1 - \rho^2} \, d\tilde{W}_1(t),
\]
with \( \tilde{W}_1(t)\) being an independent Brownian motion of both \( W(t) \) and \( \tilde{W}(t) \) and $\rho\in(-1,1)$.
So, after replacing \( dW(t) \) in \( dX(t)\), the system becomes:
\[
dX(t) = -\alpha_X X(t) \,dt + \beta_X \rho \,d\tilde{W}(t) + \beta_X \sqrt{1 - \rho^2} \,d\tilde{W}_1(t).
\]
This forms the basis for the one-factor price and volume model.
 \[
  \begin{aligned}
   X(t) &= e^{-\alpha_X t} X(0) +  \int_0^t e^{\alpha_X (s-t)} \left( \beta_X \rho \, d\tilde{W}(s) + \beta_X \sqrt{1 - \rho^2} \, d\tilde{W}_1(s) \right)
 \\
    Y(t) &= Y_0 e^{-\alpha_Y t} + \eta \int_0^t e^{-\alpha_Y (t-s)}\,d\tilde{W}(s).
  \end{aligned}
  \]
The processes \( X\) and \( Y \) are explicitly dependent through the common noise factor $\tilde{W}$.  
Let \( g \) denote the probability density function of the couple \( (X,  Y) \).   
Let $(X(T),Y(T))$ be bivariate normal with means 
\(\mu_X, \mu_Y\) and covariance matrix with entries 
\(\Sigma_{11} = \mathrm{Var}(X(T))\), 
\(\Sigma_{22} = \mathrm{Var}(Y(T))\), 
\(\Sigma_{12} = \mathrm{Cov}(X(T),Y(T))\).  
Set \(D = \Sigma_{11}\Sigma_{22} - \Sigma_{12}^2\).
The density $g$ is given explicitly by 
\[
g(x,y) \;=\; \frac{1}{2\pi \sqrt{D}} \,
\exp\!\Bigg(
-\frac{1}{2D}\Big[
\Sigma_{22}(x-\mu_X)^2 
- 2\Sigma_{12}(x-\mu_X)(y-\mu_Y) 
+ \Sigma_{11}(y-\mu_Y)^2
\Big]
\Bigg),
\]
and thus the log-density is 
\[
\ln g(x,y) \;=\; -\tfrac{1}{2}\ln(2\pi) - \tfrac{1}{2}\ln D 
-\frac{1}{2D}\Big[
\Sigma_{22}(x-\mu_X)^2 
- 2\Sigma_{12}(x-\mu_X)(y-\mu_Y) 
+ \Sigma_{11}(y-\mu_Y)^2
\Big].
\]
The derivative of $\ln g$ with respect to $x$ is
\[
\frac{\partial}{\partial x}\ln g(x,y) 
= -\frac{\Sigma_{22}(x-\mu_X) - \Sigma_{12}(y-\mu_Y)}{D}.
\]

We use the Schwartz model in \cite{benth2010computation, pricesevidence}, defined by
\begin{equation}
    S(T)=S(0) \exp(X(T) ).
\end{equation}
Hence, by letting $n=1$ we get $S(T)=h_1(T,X(T) )=S(0) \exp(X(T))=\exp(\ln S(0)+X(T))$.
We do the same thing for the volume process $V(T)$,  
\begin{equation}
    V(T)=V(0)\exp(Y(T))=\exp(\ln V(0)+Y(T) ).
\end{equation}
By letting $m=1$, we get $V(T)= h_2(Y(T) )=V(0)\exp(Y(T) )=\exp(\ln V(0)+Y(T) )$.
The payoff function can be represented as:
\begin{align}
    h(S(T), V(T))&=h(\exp(\ln S(0)+X(T)),\exp(\ln V(0)+Y(T) ))\nonumber\\
    &=G(T, \ln S(0)+X(T),\ln V(0)+Y(T) ).
\end{align}
 In this case $ \zeta_1(s)=\zeta_2(s) =\ln(s)$.
In order  to use Proposition \ref{prop 2.1} , we should verify the assumptions in I). 
The derivative with respect to $x$ of $g$ is given by.
\[
\frac{\partial g}{\partial x}(x,y)
= -\frac{\Sigma_{22}(x-\mu_X) - \Sigma_{12}(y-\mu_Y)}{D} \, g(x,y).
\]
There exists a constant $C>0$ such that
\[
\Big|\frac{\partial g}{\partial x}(x,y)\Big|
\le  C(1 + |x| + |y|)\, g(x,y).
\]

Then
\begin{align*}
    &\left| 
    G(T,x + \zeta_{1}\big(S(0)),   y + \zeta_{2}\big(V(0)))\,
    \frac{\partial}{\partial x} g(x - z,  y) 
    \right|\\&\qquad\qquad
    \leq C \times G(T,x+ \zeta_{1}\big(S(0)),  y + \zeta_{2}\big(V(0))\,
     (1 + |x-z| + |y|)\, g(x-z,y) ).
    \end{align*}
    Using the Cauchy-Schwarz inequality 
    \begin{align*}
        &\int_{\mathbb{R}^2}  G(T,x+ \zeta_{1}\big(S(0)),  y + \zeta_{2}\big(V(0))\,
     (1 + |x-z| + |y|)\, g(x-z,y) dx dy \\&\qquad\leq (\int_{\mathbb{R}^2}  G^2(T,x+ \zeta_{1}\big(S(0)),  y + \zeta_{2}\big(V(0))g(x-z,y)dxdy)^\frac{1}{2} \\
&\qquad\qquad\times(\int_{\mathbb{R}^2}  (1 + |x-z| + |y|)^2  g(x-z,y)dxdy)^\frac{1}{2}\\ &\qquad=\mathbb{E} \Big[ |G(T,X(T)+z +\zeta_{1}(S(0)),Y (T)+\zeta_{2}(V(0))|^2\Big]^\frac{1}{2} \\
&\qquad\qquad\times(\int_{\mathbb{R}^2}  (1 + |x| + |y|)^2  g(x,y)dxdy)^\frac{1}{2} \qquad< \infty.
    \end{align*}
    
Therefore, according to Proposition \ref{prop 2.1}, the delta with respect to  $S(0)$ is given by:  
\begin{align*}
\Delta_S 
&= \frac{\partial C}{\partial S(0)} \\
&=-e^{-rT}\zeta_1'(S(0))\mathbb{E} \Big[ G\big(T,X (T) + \zeta_1(S(0)),Y (T) +\ln V(0) \big)
\times \frac{\partial}{\partial x} \ln g\big(X (T),\, Y (T) \big) \Big] \\
&=e^{-rT}\mathbb{E} \Big[ G\big(T,X (T) + \ln S(0), Y(T)+\ln V(0)\big) 
   \frac{\Sigma_{22}(X(T)-\mu_X) - \Sigma_{12}(Y(T)-\mu_Y)}{D\times S(0)}   \big].
\end{align*}

 Note that in the case of wind and solar irradiation, the values of the produced power volume are always positive. This is why we apply the exponential function to model the process $V$. 
 In contrast, in the next example we consider a contract whose payoff depends on temperature, which is a proxy for power demand. Since temperature can take both positive and negative values, applying an exponential transformation would not be appropriate. Consequently, the choice of the function \( h \) must be adapted to this setting to accurately reflect the sign and magnitude of temperature variations.

\subsection{ A correlated two-factor price and one-factor temperature model}\label{sec}

We introduce Brownian motions $W_1$, $W_2$, $W_3$ with covariance structure
\[
\langle W_i, W_j \rangle_t = \rho_{ij} \, t, \quad 
\rho_{ii} = 1, \quad \rho_{ij} = \rho_{ji} \in [-1,1], \quad i,j=1,2,3.
\]
We assume that the $3\times3$ correlation matrix defined by the $\rho_{ij}$'s is positive definite. 
Consider the following system of Ornstein-Uhlenbeck stochastic differential equations:
\begin{align*}
dX_1(t) &= -\alpha_1 X_1(t)\, dt + b_1 \, dW_1(t), \\
dX_2(t) &= -\alpha_2 X_2(t)\, dt + b_2 \, dW_2(t), \\
dY(t)   &= -\alpha_3 Y(t)\, dt + b_3 \, dW_3(t),
\end{align*}
where $\alpha_i > 0$ and $b_i \in \mathbb{R}$ for $i=1,2,3$.

According to Proposition \ref{prop 2.1}, the delta with respect to  $S(0)$  is given by:
\begin{align*}
\Delta_S
&= \frac{\partial C}{\partial S(0)} \\
&=e^{-rT}\mathbb{E} \Big[ G(T,X_1(T)+z +\zeta_{1}(S(0)),X_2(T),Y (T)+\zeta_{2}(V(0))\big) \\
&\quad \times \frac{\eta_{11} (X_1 - \mu_1) +\eta_{12} (X_2 - \mu_2) + \eta_{13} (Y - \mu_3)}{S(0)} \Big].  
\end{align*}
The details and the verification of the required conditions are provided in Appendix \ref{proof of  example 2 density method}.

 \subsection{ Continuous-time autoregressive moving-average processes  (CARMA) }

We introduce here the continuous-time version of the ARMA models considered for weather modeling by \cite[Ch. 3]{benth2012modeling}. 
This class of models is called \emph{continuous-time autoregressive moving-average}, or CARMA for short, 
and was first introduced by \cite{doob1944elementary}. We base our presentation of these processes on \cite{brockwell2001levy},  
who introduced the CARMA processes to financial applications.
Let $W_1$ and $W_2$   be Brownian motions correlated such that 
\[
\operatorname{Corr}(W_1, W_2) = \rho, \qquad -1 < \rho < 1.
\]  
A positive correlation is often natural in weather applications  since many weather variables often respond to the same common driving factors (temperature systems, pressure fronts, seasonal effects, etc.).
When one location or variable is “high,” nearby or related variables tend to also be high.
We introduce the stochastic process $Z_1(t)$ and $Z_2(t)$ with values in $\mathbb{R}^p$ for $p \geq 1$ as the solution 
of the stochastic differential equations
\begin{equation} \label{eq:carma_sde1}
  dZ_1(t) = A_1 Z_1(t)\, dt + e_p \, \sigma_1(t) \, dW_1(t), 
\end{equation}
\begin{equation} \label{eq:carma_sde2}
  dZ_2(t) = A_2 Z_2(t)\, dt + e_p \, \sigma_2(t) \, dW_2(t), 
\end{equation}
where    $e_k \in \mathbb{R}^p$, $k=1, \dots, p$, is the $k$th standard Euclidean basis vector of $\mathbb{R}^p$. 
The $p \times p$-matrices $A_i, i = 1, 2$ are given by
 
\[
A_1 = 
\begin{bmatrix}
0 & 1 & 0 & \cdots & 0 \\
0 & 0 & 1 & \cdots & 0 \\
\vdots & \vdots & \vdots & \ddots & \vdots \\
0 & 0 & 0 & \cdots & 1 \\
-\alpha_p & -\alpha_{p-1} & \cdots & -\alpha_2 & -\alpha_1
\end{bmatrix}, \quad
A_2 =
\begin{bmatrix}
0 & 1 & 0 & \cdots & 0 \\
0 & 0 & 1 & \cdots & 0 \\
\vdots & \vdots & \vdots & \ddots & \vdots \\
0 & 0 & 0 & \cdots & 1 \\
-\alpha'_p & -\alpha'_{p-1} & \cdots & -\alpha'_2 & -\alpha'_1
\end{bmatrix}.
\]

The constants $\alpha_k$,$\alpha'_k$, $k=1, \dots, p$, are assumed to be non-negative with $\alpha_p > 0$ and $\alpha'_p > 0$ . 
 To model it, we consider a bounded and continuous function 
$\sigma_1 , \sigma_2 : \mathbb{R}_{+} \mapsto \mathbb{R}$ in the dynamics of $Z_1$ and $Z_2$, where $\sigma_1(t)$ and $\sigma_2(t)$ are strictly bounded away 
from zero, i.e., there exists two  constants $\overline{\sigma_1} > 0,\overline{\sigma_2} > 0$ such that 
\[
  \sigma_1(t) \geq \overline{\sigma_1}  ,   \sigma_2(t) \geq \overline{\sigma_2}\quad \text{for all } t \geq 0.
\]
In the sequel, we use the notation $I_n$ for the $n\times n$ identity matrix. 
Sometimes we simply write $I$ if the dimension is clear from the context. 
The Lemma 4.1 in \cite{benth2012modeling} proves that the matrices $A_i$, $i=1,2$,  are invertible.
By Lemma 4.3 in \cite{benth2012modeling} the solution of  \ref{eq:carma_sde1} and \ref{eq:carma_sde2} are given by
\[
Z_1(s) = e^{A_1 (s - t)} Z_1(t) 
  + \int_{t}^{s} e^{A_1 (s - u)} e_p \, \sigma_1(u) \, dW_1(u).   \qquad s \geq t\geq 0.
\]
\[
Z_2(s) = e^{A_2 (s - t)} Z_2(t) 
  + \int_{t}^{s} e^{A_2 (s - u)} e_p \, \sigma_2(u) \, dW_2(u).   \qquad s \geq t\geq 0.
\]
Now for $0 \leq q < p$, we define the vectors $\mathbf{b^1}$ and $\mathbf{b^2}$ $\in \mathbb{R}^p$ with coefficients 
$b^1_j$,$b^2_j$, $j = 0, 1, \dots, p-1$, satisfying $b^1_q =b^2_q = 1$ and $b^1_j =b^2_j= 0$ for $q < j < p$. 
We define the CARMA$(p,q)$ process as
\begin{equation}
X(t) = \mathbf{b^1}^\top Z_1(t)
\qquad\text{and}\qquad
Y(t) = \mathbf{b^2}^\top Z_2(t),
\label{eq:carma_outputs}
\end{equation}
where \(X(t)\) is the price factor and \(Y(t)\) is the volume factor.
Starting from 0, $X(t)$ and $Y(t)$ are  given by:
\begin{equation}
  X(t) = \mathbf{b^1}^\top e^{A_1 t} Z_1(0) 
  + \int_0^t \mathbf{b^1}^\top e^{A_1(t-u)} e_p \, \sigma_1(u) \, dW_1(u),
  \label{eq:carma_mild}
\end{equation}
\begin{equation}
  Y(t) = \mathbf{b^2}^\top e^{A_2t} Z_2(0) 
  + \int_0^t \mathbf{b^2}^\top e^{A_2(t-u)} e_p \, \sigma_2(u) \, dW_2(u).
  \label{eq:carma_mild2}
\end{equation}
We start by computing the expectations of $X(t)$ and $Y(t)$:
\[ 
\mathbb{E}[X(T)] = \mathbf{b^1}^\top e^{A_1 T}\, \mathbb{E}[Z_1(0)]
\qquad\text{and}\qquad
\mathbb{E}[Y(T)] = \mathbf{b^2}^\top e^{A_2 T}\, \mathbb{E}[Z_2(0)].
\]
Next, the variances and the covariance, taking into account the correlation $\rho$ between $W_1$ and $W_2$:\\
\(\sigma_X^2=
\mathrm{Var}[X(T)] = \int_0^T \!\left( \mathbf{b^1}^\top e^{A_1(T-u)} e_p \right)^2 \sigma_1^2(u)\, du\),\\
 \( \sigma_Y^2= 
\mathrm{Var}[Y(T)] = \int_0^T \!\left( \mathbf{b^2}^\top e^{A_2(T-u)} e_p \right)^2 \sigma_2^2(u),\, du
\)\\
\(\mathrm{Cov}(X(T),Y(T)) = \rho \int_0^T \left( \mathbf{b^1}^\top e^{A_1(T-u)} e_p \right) \left( \mathbf{b^2}^\top e^{A_2(T-u)} e_p \right) \sigma_1(u) \sigma_2(u) \, du.\)

Therefore, the pair $(X,Y)$ is Gaussian, and its bivariate density is given by: 
\begin{align*}
    g_{X,Y}(x,y)& = \frac{1}{2\pi \sigma_X \sigma_Y \sqrt{1-\rho_{XY}^2}} \\
&\times \exp \Bigg\{ -\frac{1}{2(1-\rho_{XY}^2)} 
\Big[ \frac{(x-\mu_X)^2}{\sigma_X^2} - 2 \frac{\rho _{XY}(x-\mu_X)(y-\mu_Y)}{\sigma_X \sigma_Y} + \frac{(y-\mu_Y)^2}{\sigma_Y^2} \Big] \Bigg\},
\end{align*}
where \[
\rho_{XY}
= \frac{\operatorname{Cov}(X(T),Y(T))}
{\sqrt{\operatorname{Var}(X(T))}\,\sqrt{\operatorname{Var}(Y(T))}}.
\] 
The logarithm of the density:  
\begin{align*}
\ln g_{X,Y}(x,y) &= -\frac{1}{2} \ln \big( (2\pi)^2 \sigma_X^2 \sigma_Y^2 (1-\rho_{XY}^2) \big) \\
&\quad - \frac{1}{2(1-\rho_{XY}^2)} \Big[ \frac{(x-\mu_X)^2}{\sigma_X^2} 
- 2 \frac{\rho_{XY} (x-\mu_X)(y-\mu_Y)}{\sigma_X \sigma_Y} + \frac{(y-\mu_Y)^2}{\sigma_Y^2} \Big].
\end{align*}
Finally, the derivative of the ln-density with respect to $x$ is:
\[
\frac{\partial}{\partial x} \ln g_{X,Y}(x,y) 
= - \frac{1}{1-\rho_{XY}^2} \Bigg[ \frac{x-\mu_X}{\sigma_X^2} - \frac{\rho _{XY}(y-\mu_Y)}{\sigma_X \sigma_Y} \Bigg].
\]
For the price, we use the Schwartz model in \cite{benth2010computation, pricesevidence}, which  is defined by
\begin{equation}
    S(T)=S(0) \exp(X(T) ),
\end{equation}
so by letting $n=1$ we get $S(T)=h_1(T,X(T) )=S(0) \exp(X(T))=\exp(\ln S(0)+X(T))$.
We do the same thing for the volume process $V(t)$,  
by letting $m=1$, we get $V(T)= h_2(t,Y(T) )=V(0)\exp(Y(T) )=\exp(\ln V(0)+Y(T) )$.
In this case $ \zeta_1(s)=\zeta_2(s) =\ln(s)$.
\\The payoff function can be represented as:
\begin{align}
    h(S(T), V(T))&=h(\exp(\ln S(0)+X(T)),\exp(\ln V(0)+Y(T) ))\nonumber\\
    &=G(T, \ln S(0)+X(T),\ln V(0)+Y(T) ).
\end{align}
 In this case $ \zeta_1(s)=\zeta_2(s) =\ln(s)$.
 In order to use Proposition \ref{prop 2.1}, we should verify the assumptions in I):

We consider the derivative of the bivariate Gaussian density  $g_{X,Y}(x,y)$ with respect to $x$:
\[
\frac{\partial g_{X,Y}}{\partial x} (x,y) = g_{X,Y} (x,y) \cdot \frac{-1}{1-\rho_{XY}^2} 
\left( \frac{x-\mu_X}{\sigma_X^2} - \frac{\rho_{XY} (y-\mu_Y)}{\sigma_X \sigma_Y} \right).
\]
Taking the absolute value and using the triangle inequality, we obtain an upper bound:
\[
\left| \frac{\partial g_{X,Y}}{\partial x} (x,y) \right| 
\leq g_{X,Y}(x,y) \cdot \frac{1}{1-\rho_{XY}^2} 
\left( \frac{|x-\mu_X|}{\sigma_X^2} + \frac{|\rho_{XY}| |y-\mu_Y|}{\sigma_X \sigma_Y} \right).
\]

Observe that
\[
\frac{|x-\mu_X|}{\sigma_X^{2}}
+
\frac{|\rho_{XY}|\,|y-\mu_Y|}{\sigma_X\sigma_Y}
\leq
C_1\bigl(|x-\mu_X|+|y-\mu_Y|\bigr),
\]
where
\[
C_1=
\max\left\{
\frac{1}{\sigma_X^{2}},
\frac{|\rho_{XY}|}{\sigma_X\sigma_Y}
\right\}.
\]
Hence,
\[
\left| \frac{\partial g_{X,Y}}{\partial x}(x,y) \right|
\leq
C_2
\bigl(|x-\mu_X|+|y-\mu_Y|\bigr)
g_{X,Y}(x,y),
\]
with
\[
C_2=\frac{C_1}{1-\rho_{XY}^{2}}.
\] Then 
\begin{align*}
&\left|G(T,x + \zeta_{1}(S(0)),y + \zeta_{2}(V(0)))\,
\frac{\partial}{\partial x} g_{X,Y}(x - z,y)\right| \\
&\leq C_2 |G(T,x + \zeta_{1}(S(0)),y + \zeta_{2}(V(0)))|
\bigl(|x-z-\mu_X|+|y-\mu_Y|\bigr)
g_{X,Y}(x-z,y).
\end{align*}
Using the Cauchy–Schwarz inequality
\begin{align*}
&\int_{\mathbb{R}^2} \left|G(T,x+ \zeta_{1}(S(0)),y + \zeta_{2}(V(0)))\,\frac{\partial}{\partial x} g_{X,Y}(x - z,y)\right| \,dx\,dy  \\
&\leq C_2\Bigg(\int_{\mathbb{R}^2} G^2(T,x+ \zeta_{1}(S(0)),y + \zeta_{2}(V(0))) g_{X,Y}(x-z,y)\,dx\,dy\Bigg)^{\tfrac{1}{2}}\\
&\times\Bigg(\int_{\mathbb{R}^2} |
\bigl(|x-z-\mu_X|+|y-\mu_Y|\bigr)^2
g_{X,Y}(x-z,y)\,dx\,dy\Bigg)^{\tfrac{1}{2}} \\
&=C_2\mathbb{E} \Big[ |G(T,X(T)+z +\zeta_{1}(S(0)),Y (T)+\zeta_{2}(V(0))|^2\Big]^\frac{1}{2}\\
&\times\Bigg(\int_{\mathbb{R}^2} 
\bigl(|x-\mu_X|+|y-\mu_Y|\bigr)^2
g_{X,Y}(x,y)\,dx\,dy\Bigg)^{\tfrac{1}{2}} < \infty.
\end{align*}
In fact, we have
\[
g_{X,Y}(x,y)
\leq
K
\exp\left(
-c
\left[
\frac{(x-\mu_X)^2}{\sigma_X^2}
+
\frac{(y-\mu_Y)^2}{\sigma_Y^2}
\right]
\right),
\]
where
\[
K=
\frac{1}
{2\pi \sigma_X \sigma_Y \sqrt{1-\rho_{XY}^2}},
\qquad
c=
\frac{1}{2\left(1+|\rho_{XY}|\right)}.
\]
Consequently,
$\Bigg(\int_{\mathbb{R}^2} 
\bigl(|x-\mu_X|+|y-\mu_Y|\bigr)^2
g_{X,Y}(x,y)\,dx\,dy\Bigg)^{\tfrac{1}{2}} < \infty.$\\
According to Proposition \ref{prop 2.1}, the delta with respect to $S(0)$ is given by:   
\begin{align*}
\Delta_S &= \frac{\partial C}{\partial S(0)} \\
&=\mathbb{E}\Bigg[ G\big(T,X(T)+ \zeta_{1}(S(0)),Y(T)+\zeta_{2}(V(0))\big) \\
&\quad\times \frac{e^{-rT}}{S(0)(1-\rho^2)}\Bigg(\frac{X-\mu_X}{\sigma_X^2}-\frac{\rho_{XY} (Y-\mu_Y)}{\sigma_X \sigma_Y}\Bigg)\Bigg].
\end{align*}
Following the same technique used for the price process and verification of assumptions (II) in Proposition \ref{prop 2.1}, we can calculate the delta with respect to the initial volume $V(0)$. 

If we let $V$ model the temperature, the volume process should be on an arithmetic scale. 
We have 
\[
S(T) = S(0) e^{X(T)}, \quad V(T) = V(0)+Y(T), \quad  \zeta_2(v)=v
\]
where $(X,Y)$ is Gaussian with means $(\mu_X, \mu_Y)$, variances $(\sigma_X^2, \sigma_Y^2)$, and correlation $\rho_{XY}
$. 
 In this case, according to Proposition \ref{prop 2.1}, the delta with respect to $V(0)$ is given by:
\[
 \Delta_V = \mathbb{E}\left[  G\big(T,X(T)+ \zeta_{1}(S(0)),Y(T)+V(0)\big) \cdot \frac{e^{-rT}}{(1-\rho^2)} \left( \frac{Y(T)-\mu_Y}{\sigma_Y^2} - \frac{\rho_{XY} (X(T)-\mu_X)}{\sigma_X \sigma_Y} \right) \right].
\]

\section{Computing Deltas with  Malliavin Calculus} 
 
The Malliavin method is very useful when one does not have the density function of the multifactor model. In this section, we calculate the delta using Malliavin calculus, similar to the approach in  \cite{benth2010robustness,sole2007canonical}. The main difference is in the payoff; we use both price and volume, whereas they consider only the price.  

We introduce the generic notation $L(t)$ for a Lévy process on the given probability space and denote by $W(t)$  a Brownian motion independent of $L(t)$, with $t \in [0,T]$ and $L(0) = W(0) = 0$ by convention. We work with the RCLL version of the Lévy process and define.
\[
\Delta L(t) := L(t) - L(t^{-}).
\]
Let the Lévy measure of $L(t)$ be denoted by $\ell(dz)$. Recall that $\ell(dz)$ is a $\sigma$-finite Borel measure on $\mathbb{R}_0 := \mathbb{R} \setminus \{0\}$, and assume that  $\int_{\mathbb{R}} z^k \ell(dz)$ is finite for all $k$.

We also recall the Lévy-Itô decomposition of a Lévy process (see Sato \cite{sato1999lvy}):

\begin{theorem}
For $t \geq 0$, let $L(t)$ be a Lévy process on $\mathbb{R}$ and $\ell$ its Lévy measure. Then the following hold:
\begin{itemize}
   \item $\ell$ verifies $\int_{\mathbb{R}_0} \min(1, z^2) \, \ell(dz) < \infty$, for $z \in \mathbb{R}_0$.

    \item The jump measure of $L(t)$, denoted by $N(dt, dz)$, is a Poisson random measure on $[0,\infty[ \times \mathbb{R}_0$ with intensity measure $\ell(dz) dt$,
    \item There exists a Brownian motion $W(t)$ and two constants $a, b \in \mathbb{R}$ such that
   \begin{equation} \label{4.1}
L(t) = a t + b W(t) + Z(t) + \lim_{\epsilon \to 0} \tilde{Z}_{\epsilon}(t).
\end{equation}

    where
    \[
    Z(t) := \sum_{s \in [0,t]} \Delta L(s) \mathbf{1}_{\{ |\Delta L(s)| \geq 1 \}}
    = \int_{0}^{t} \int_{|z| \geq 1} z \, N(ds, dz).
\]

    and
    \[
    \tilde{Z}_{\epsilon}(t) :=  \sum_{s \in [0,t]} \Delta L(s) \mathbf{1}_{\{ \epsilon \leq | \Delta L(s) |< 1 \}}-t \int_{\epsilon \leq |z| < 1}z \ \ell(dz)= \int_0^t \int_{\epsilon \leq |z| < 1} z \,\tilde{N}(ds, dz),
    \]
    where $\tilde{N}$ denotes the compensated Poisson random measure of $L(t)$.
\end{itemize}

The convergence of $\tilde{Z}_{\epsilon}(t)$ in the above decomposition is almost sure and uniform over $t \in [0,T]$. The components $W(t)$, $Z(t)$, and $\tilde{Z}_{\epsilon}(t)$ are independent.

\end{theorem}
In various applications involving statistical and numerical methods, it is often beneficial to approximate small jumps using a scaled Brownian motion. This approach was initially suggested by Rydberg~\cite{rydberg1997normal} as a technique for simulating the trajectory of a L\'evy process with NIG-distributed increments and was later examined in detail by Asmussen and Rosi\'nski~\cite{asmussen2001approximations}. We will adopt this method to study the
deltas in jump-diffusion models.

We introduce the following notation for the variation of the Lévy process \( L(t) \) near the origin:
\begin{equation}
    \sigma^2(\varepsilon) := \int_{|z|<\varepsilon} z^2 \ \ell(dz), \quad 0 < \varepsilon \leq 1.
    \label{eq:sigma_variance}
\end{equation}
Since every Lévy measure \( \ell(dz) \) integrates \( z^2 \) in an open interval around zero, it follows that \( \sigma^2(\varepsilon) \) is finite for any \( \varepsilon > 0 \). Note that \( \sigma^2(\varepsilon) \) represents the variance of the jumps smaller than \( \varepsilon \) of \( L(t) \), assuming the process is symmetric and has zero mean. By the dominated convergence theorem, \( \sigma^2(\varepsilon) \) converges to zero as \( \varepsilon \to 0 \).

We recall the Lévy-Itô decomposition of a Lévy process \( L(t) \) and introduce the approximate Lévy process (in law):
\begin{equation}
    L_\varepsilon(t) := at + bW(t) + \sigma(\varepsilon) B(t) + Z(t) + \tilde{Z}_\varepsilon(t),
    \label{eq:approx_levy}
\end{equation}
where \( \sigma^2(\varepsilon) \) is given in \eqref{eq:sigma_variance}, and \( B(t) \) is a Brownian motion independent of \( L(t) \) (hence also independent of \( W(t) \)). From the definition of \( \tilde{Z}_\varepsilon(t) \), we see that the small jumps (compensated by their expectation) in \( L(t) \) have been replaced by a Brownian motion scaled by \( \sigma(\varepsilon) \), which is the standard deviation of the compensated small jumps.

We now state the following result from   \cite{benth2010robustness}:  
  
\begin{proposition} 
Let the processes \( L(t) \) and \( L_\varepsilon(t) \) be defined as in equations (3.1) and \eqref{eq:approx_levy}, respectively. Then, for every \( t \geq 0 \),
\[
    \lim_{\varepsilon \to 0} L_\varepsilon(t) = L(t) \quad P\text{-a.s.}
\]
Moreover, the convergence also holds in \( L^1(\Omega, \mathcal{F}, P) \) with:
\[
    \mathbb{E}[|L_\varepsilon(t) - L(t)|] \leq 2\sigma(\varepsilon) \sqrt{t}.
\]
\end{proposition}

 \subsection{Chaotic representation for L\' evy processes and Malliavin derivative.}

In the work of It\^{o} \cite{ito1956spectral}, multiple stochastic integrals with respect to a Poisson random measure are introduced. For a more general treatment extending to random measures with independent values, see Di Nunno \cite{di2004orthogonal}. The construction follows a similar approach as in the Wiener case, as detailed in Kuo \cite{kuo2006stochastic}.

Here and in the sequel we assume that the L\'{e}vy measure satisfies  
\begin{equation}
    \sigma^2(\infty) := \int_{\mathbb{R}_0} z^2 \ell(dz) < \infty.
\end{equation}

Consider a L\'{e}vy process $L$ having a representation as in \ref{4.1} with $b = 1$.
Introduce the measure $M$ on the Borel $\sigma$-algebra $\mathcal{B}(\mathbb{R}_+ \times \mathbb{R})$ such that for $E \in \mathcal{B}(\mathbb{R}_+ \times \mathbb{R})$,  
\begin{equation}
    M(E) = \lambda(E^{(0)}) + \int_{E'} z^2 dt \ell(dz),
\end{equation}
where $E^{(0)} = \{t \in \mathbb{R}_+; (t,0) \in E\}$ and $E'= E -\{(t,0) \in E\}$. 

Define  
\begin{equation}
    \mu(E) = \int_{E^{(0)}} dW(t) + \lim_{n \to \infty} \sum_{ \ {(t,z) \in E, \frac{1}{n} < |z| < n }\ }z N(dt,dz),
\end{equation}
where the limit is taken in $L^2(\Omega)$. 

Denote by $L^2(\Omega) := L^2(\Omega, \mathcal{F}, P)$ the Hilbert space of square-integrable random variables, equipped with the norm $\|F\|_2 = (\mathbb{E} [F^2])^{1/2} < \infty$. The set function $\mu$ is a centered random measure such that for $E_1, E_2 \in \mathcal{B}(\mathbb{R}_+ \times \mathbb{R})$ with $M(E_1) < \infty$ and $M(E_2) < \infty$,  
\begin{equation}
    E[\mu(E_1) \mu(E_2)] = M(E_1 \cap E_2).
\end{equation}

Denote by $L^2_n = L^2((\mathbb{R}_+ \times \mathbb{R})^n, \mathcal{B}((\mathbb{R}_+ \times \mathbb{R})^n), M^{\otimes n})$, with the standard norm $|\cdot|_n$. Let  
\begin{equation}
    f = 1_{E_1 \times \dots \times E_n},
\end{equation}
where the sets $E_1, \dots, E_n \in \mathcal{B}(\mathbb{R}_+ \times \mathbb{R})$ are pairwise disjoint and,
\begin{equation}
M(E_1) < \infty, \dots, M(E_n) < \infty.
\end{equation}

The multiple stochastic integral of the elementary function $f$ is an element in $L^2(\Omega)$ defined as follows  
\begin{equation}
    I_n(f) := \int_{(\mathbb{R}_+ \times \mathbb{R})^n} f \mu^{\otimes n} := \mu(E_1) \cdots \mu(E_n).
\end{equation}

By standard arguments, $I_n$ can be extended to the symmetric function in $L^2_n$ by appealing to linearity and continuity. Moreover, for any symmetric functions $f \in L^2_n$ and $g \in L^2_m$ we have  
\begin{equation}
    E[I_n(f) I_m(g)] = \delta_{n,m} n! \int_{(\mathbb{R}_+ \times \mathbb{R})^n} \tilde{f}\tilde{g} dM^{\otimes n},
\end{equation}
where $\delta_{n,m} = 1$ if $n = m$ and $0$ otherwise.

It\^{o} \cite{ito1956spectral} proves the following chaos expansion for elements of $L^2(\Omega)$:
\begin{theorem}
    For any $F \in L^2(\Omega)$ there exists a unique sequence $(f_n)_{n \geq 0}$ of symmetric functions $f_n \in L^2_n$ such that  
    \begin{equation}
        F = \sum_{n=0}^{\infty} I_n(f_n),
    \end{equation}
    where the series is converging in $L^2(\Omega)$. Moreover, it holds that  
    \begin{equation}
        \|F\|^2_2 = \sum_{n=0}^{\infty} n! |f_n|^2_n.
    \end{equation}
\end{theorem}

Note that, among all the stochastic measures with independent values in $L^2(\Omega)$, it is only in the case of mixtures of Gaussian and Poisson measures that it is possible to achieve chaos representation type results. This is proved in Theorem 2.2 in Di Nunno \cite{di2004orthogonal}.

In the work of Sol\'e, Utzet, and Vives \cite{sole2007canonical}, a stochastic derivative is introduced within a subspace of $L^2(\Omega)$. This derivative exploits chaos expansion representations, in a manner analogous to the Malliavin derivative in Wiener space (see Nualart \cite{nualart2006malliavin}). Let $F \in L^2(\Omega)$ be expressed through its chaotic representation:
\begin{equation}
F = \sum_{n=0}^{\infty} I_n(f_n),
\end{equation}
where the symmetric functions $f_n$ satisfy the summability condition:
\begin{equation}
\sum_{n=1}^{\infty} n n! \| f_n \|^2_{n} < \infty.
\end{equation}
The Malliavin derivative $D F : \mathbb{R}^+ \times \mathbb{R} \times \Omega \to \mathbb{R}$ of $F$ is then defined as
\begin{equation}
D_w F = \sum_{n=1}^{\infty} n I_{n-1} ( f_n (w, \cdot) ), \quad w \in \mathbb{R}^+ \times \mathbb{R},
\end{equation}
where convergence occurs in $L^2(\mathbb{R}^+ \times \mathbb{R} \times \Omega, M \otimes P)$. The operator $D$ can be interpreted as an annihilation operator, shifting the chaos expansion index by one unit to the left.

Denote by $\text{Dom}(D)$ the space of functionals $F \in L^2(\Omega)$ that satisfy the summability condition (2). This space is a Hilbert space under the inner product:
\begin{equation}
\langle F, G \rangle = \mathbb{E}[FG] + \mathbb{E} \left[ \int_{\mathbb{R}^+ \times \mathbb{R}} D_w F D_w G M(dw) \right],
\end{equation}
where $D$ is a closed operator mapping $\text{Dom }(D)$ to $L^2(\mathbb{R}^+ \times \mathbb{R} \times \Omega, M \otimes P)$.

Further, let $\text{Dom}(D_0)$ be the space of random variables $F = \sum_{n=0}^{\infty} I_n(f_n) \in L^2(\Omega)$ satisfying:
\begin{equation}
\sum_{n=1}^{\infty} n n! \int_{\mathbb{R}^+ \times (\mathbb{R}^+ \times \mathbb{R})^{n-1}} f_n^2 ((t,0), w_1, \dots, w_{n-1}) dt dM^{\otimes(n-1)} (w_1, \dots, w_{n-1}) < \infty.
\end{equation}
For $F \in \text{Dom} (D_0)$, we define the stochastic process:
\begin{equation}
D_{t,0} F = \sum_{n=1}^{\infty} I_{n-1} ( f_n ((t,0), \cdot) ),
\end{equation}
where the sum converges in $L^2(\mathbb{R}^+ \times \Omega, dt \otimes P)$.

Similarly, for $\ell(dz) \neq 0$, define $\text{Dom} (D_J)$ as the set of $F = \sum_{n=0}^{\infty} I_n(f_n) \in L^2(\Omega)$ such that:
\begin{equation}
\sum_{n=1}^{\infty} n n! \int_{(\mathbb{R}^+ \times \mathbb{R}_0) \times (\mathbb{R}^+ \times \mathbb{R})^{n-1}} f_n^2 ((t, z), w_1, \dots, w_{n-1}) dM^{\otimes(n)} (w_1, \dots, w_{n-1}) < \infty.
\end{equation}
For $F \in \text{Dom} (D_J)$, we define the stochastic derivative:
\begin{equation}
D_{t,z} F = \sum_{n=1}^{\infty} I_{n-1} ( f_n ((t,z), \cdot) ),
\end{equation}
with convergence in $L^2(\mathbb{R}^+ \times \mathbb{R}_0 \times \Omega, z^2 dt d\ell(z) \otimes P)$. The operator $D_{t,0}$ is closely related to differentiation with respect to the Brownian part of the process $L$, allowing the application of classical Malliavin calculus techniques in many settings.

Let \((\Omega_W, \mathcal{F}_W, P_W)\) and \((\Omega_J, \mathcal{F}_J, P_J)\) be the canonical spaces for the Brownian motion and pure jump Lévy process, respectively. We can interpret
\[
\Omega = \Omega_W \times \Omega_J, \quad \mathcal{F} = \mathcal{F}_W \otimes \mathcal{F}_J, \quad P = P_W \otimes P_J.
\]

The following chain rule for \(D_{t,0}\) is proved by Solé, Utzet, and Vives \cite{sole2007canonical}.

\begin{proposition}
    
 Assume \(F = f(Z, Z_0) \in L^2(\Omega_W \times \Omega_J)\), with \(Z \in \text{Dom} (D_W)\), \(Z_0 \in L^2(\Omega_J)\), and \(f(x, y)\) being a continuously differentiable function with bounded partial derivative in the first variable. Then \(F \in \text{Dom} (D_0)\), and
\[
D_{t,0} F = \frac{\partial f}{\partial x}(Z, Z_0) D_W Z,
\]
where \(D_W\) is the Malliavin derivative in \((\Omega_W, \mathcal{F}_W, P_W)\) and \(\text{Dom} (D_W)\) its domain.
\end{proposition}

\subsubsection{Skorohod integral and duality formula}

In Solé, Utzet, and Vives \cite{sole2007canonical}, the Skorohod integral with respect to a mixture of Gaussian and Poisson random measures is also defined (see Di Nunno \cite{di2004orthogonal} and Di Nunno and Rozanov \cite{di2007stochastic} for the treatment with respect to general stochastic measures in \(L^2(\Omega)\)). Let us consider  
\[
A(w) = \sum_{n=0}^{\infty} I_n(\hat{f_n}(w, \cdot)), \quad w \in \mathbb{R}^+ \times \mathbb{R},
\]
where   \(f_n \in L^2_{n+1}\) is symmetric in the last \(n\) variables and \(\hat{f_n}\) is the symmetrization of \(f_n\) in all \(n + 1\) variables. If  
\begin{equation}\label{prp1}
\sum_{n=0}^{\infty} (n + 1)! \, |\hat{f_n}|^2_{n+1} < \infty,
\end{equation}
the Skorohod integral of \(A(w)\), \(w \in \mathbb{R}^+ \times \mathbb{R}\), is defined by
\[
\delta(A) := \sum_{n=0}^{\infty} I_{n+1}(\hat{f_n}),
\]
where the convergence of the series on the right-hand side is in \(L^2(\Omega)\). Denote by \(\text{Dom}(\delta)\) the set of random fields \(A(w)\) satisfying \ref{prp1}. The following is a duality formula proven by Solé, Utzet, and Vives \cite{sole2007canonical}:

\begin{proposition}
    
 Let \(A \in L^2(\mathbb{R}^+ \times \mathbb{R} \times \Omega, \mu \otimes P)\). The random field \(A\) belongs to \(\text{Dom}(\delta)\) if and only if there is a constant \(C\) such that for all \(F \in \text{Dom}(D)\),
\[
|E\left[\int_{\mathbb{R}^+ \times \mathbb{R}} A(w) D_{w} F M(dw)\right]| \leq C \|F\|^2.
\]
If \(A \in \text{Dom} (\delta)\), then \(\delta(A)\) is the element of \(L^2(\Omega)\) characterized by
\[
E[\delta(A) F] = E\left[\int_{\mathbb{R}^+ \times \mathbb{R}} A(w) D_{w} F M(dw)\right],
\]
for any \(F \in \text{Dom}(D)\).

\end{proposition}
\subsection{Delta with respect to the price}

Let $S(t)$ be the price process
\begin{equation} \label{price}
    \begin{cases}
        dS(t) = \alpha(S(t)) dt + \sigma(S(t)) dW (t) + \int_{\mathbb{R}^{*}} \varphi(S(t), z) \widetilde{N}(dt, dz) \\ 
        S(0) = x ,
    \end{cases}
\end{equation}
  and $V(t)$ be the volume process
\begin{equation}
    \begin{cases}
        dV (t) = a(V(t)) dt + b(V(t)) dW_2(t)  + \int_{\mathbb{R}^{*}} c(V(t), z) d\widetilde{N}(dt, dz) \\ 
        V (0) = v, 
    \end{cases}
\end{equation}
where \( \alpha, \sigma, \varphi, a, b, \text{ and } c \) are measurable functions such that for all \(t > 0\) and \(z \in \mathbb{R}^{*}\):
\begin{equation*}
    \int_{\mathbb{R}_{+}} \left[ \left| \alpha(S(t) \right| + \sigma^{2}(S(t) + \int_{\mathbb{R}^{*}} \varphi(S(t), z) \nu(dz) \right] dt < \infty
\end{equation*}
and
\begin{equation*}
    \int_{\mathbb{R}_{+}} \left[ \left| a(S(t)) \right| + b^{2}(S(t)) + \int_{\mathbb{R}^{*}} c(S(t), z)\nu(dz) \right] dt < \infty,
\end{equation*}
$a.s.$ Furthermore, $W(t)$ and $W_2(t)$ are correlated by the relation 
\begin{equation*}
    dW(t)=\rho dW_2(t)+ \sqrt{1-\rho^2}  dW_1(t)
\end{equation*}
with $W_1(t)$ is independent  of both  $ W(t)$   and $ W_2(t)$.
The full system governing the dynamics of \( X(t) \) and \( V(t) \) is given by

     \begin{align}\label{stt}
         dS(t) =& \alpha(S(t))\, dt + \sigma(S(t))  \rho\, dW_2(t) + \sigma(S(t))\sqrt{1 - \rho^2}\, dW_1(t) \nonumber  \\  &+ \int_{\mathbb{R}^{*}} \varphi(S(t), z)\, \widetilde{N}(dt, dz),
     \end{align}
     and
       
       \begin{equation*}\\ 
            dV (t) = a(V(t)) dt + b(V(t)) dW_2(t)  + \int_{\mathbb{R}^{*}} c(V(t), z) \widetilde{N}(dt, dz).
\end{equation*}
  The payoff function is given by 
\begin{equation}
      h(S(t),V(t)) = \iota(S(t))\times  \kappa(V(t))
\end{equation}
 where $S$ is the price and $V$ is the wind speed  
and $\iota(S(t))\times  \kappa(V(t))\in L^2(\Omega)$.

Let $  \mathcal{F}_{t}^{\widetilde{N}}= \sigma   \left \{ \int_{0}^{s}\int_{A} \tilde{N}(du,dz) ; s\leq t , A\in\mathcal{B}(\mathbb{R}^{*}) \right \}$. Assume that  $\alpha ,\sigma , a$ and  $b  $ are continuously differentiable functions with bounded derivatives. Following \cite{benth2010robustness}, we consider now
 Markov jump-diffusions $S$ of the form \eqref{price} for which we have  a continuously differentiable function $\psi_1 $ with bounded derivatives in the first variable such that  
\begin{equation}\label{CD}
\begin{cases}
S(t) =\psi_1(S^{ c}(t),S^{d}(t))\\ 
S(0)=x.  
\end{cases}
\end{equation}
Here $S^{ c}$ satisfies the stochastic differential equation 
\begin{equation} \label{eq18N}
\begin{cases}
dS^{c}(t)=  \alpha_c(S^c(t))\, dt + \sigma_c(S^c(t))  \rho\, dW_2(t) + \sigma_c(S^c(t))\sqrt{1 - \rho^2}\, dW_1(t)\\ 
S^{c}(0)=x,
\end{cases}
\end{equation}
with $\alpha _{c}$ and  $\sigma_{c}$  are continuously differentiable functions. Note that if we deal only with the continuous part of \eqref{price} then $\alpha$ is denoted by $\alpha_c$ and $\sigma$  by $\sigma_c$. 
$S^{d}$ is adapted to the natural filtration $\mathcal{F}^{\tilde{N}}$ of the compensated compound Poisson process $\tilde{N}$. In particular, $S^{d}$ does not depend on $x$. Here $S^d$ satisfies the SDE :

\begin{equation}
    dS^d(t)= \int_{\mathbb{R}^{*}} \varphi(S(t), z)\, N(dt, dz).
\end{equation}

The jump-diffusion process of type  \eqref{CD} is called \textit{separable}.

The solution of  \eqref{eq18N} satisfies:  
\begin{equation*}
     S^{c}(t)=x+\int_{0}^{t}\alpha _{c}(S ^c(s) ) ds +\rho\int_{0}^{t} \sigma _{c}(S^c(s) ) dW_2(s)+\sqrt{1 - \rho^2}\int_{0}^{t} \sigma _{c}(S ^c(s) ) dW_1(s).
    \end{equation*}
 
Assume that there exist  measurable processes $v_1$ and $v_2$ verifying $a.s.$ $\int_0^T|v_1(s) |ds<\infty$ and  $\int_0^Tv_2^2(s) ds<\infty$, such that 
 for almost every \(s \in [0,T]\) and $a.s.$,
\[
\frac{\partial}{\partial x} \alpha_c(S^c(s))\le v_1(s),\;\; 
\frac{\partial}{\partial x} \sigma_c(S^c(s))\le v_2(s),\;\;  \\ 
\]
\medskip
\begin{remark}
    
Since $\alpha_c$ and $\sigma_c$ are $C^1$ and $S^c$ is  differentiable with respect to $x$ , the compositions $\alpha_c(S^c(\cdot))$ and $\sigma_c(S^c(\cdot))$ are almost everywhere differentiable, and by the chain rule
\[
\frac{\partial}{\partial x} \alpha_c(S^c (s)) = \alpha_c'(S^c (s)) \cdot \frac{\partial S^c (s)}{\partial x}, 
\qquad 
\frac{\partial}{\partial x}  \sigma_c(S^c (s)) = \sigma_c'(S^c (s)) \cdot \frac{\partial S^c(s)}{\partial x}.
\]
So, the integrals of derivatives are well defined.
By \cite{ikeda2014stochastic}, we will apply the derivative with respect to $x$:  
 \begin{align} \label{FIVA1}
     \frac{\partial S^{c}(t)}{\partial x}&=1+\int_{0}^{s}\alpha _{c}'(S^c  (s) )\times \frac{\partial S^{c}(s)}{\partial x} ds   + \rho\int_{0}^{t} \sigma  _{c}'(S^c  (s) )\times \frac{\partial S^{c}(s)}{\partial x} dW_2(s)  \\
     &\qquad+  \sqrt{1 - \rho^2} \int_{0}^{t} \sigma  _{c}'(S  ^c(s) )\times \frac{\partial S^{c}(s)}{\partial x} dW_1(s).\nonumber
 \end{align}
We associate to the process $S^{c}$ a process $Z_{1}$ called the first variation, such that  
\begin{equation*}
  Z_{1}(t)= \frac{\partial S^{c}(t)}{\partial x}.  \end{equation*}
In the following theorem, we provide the delta with respect to the price for payoff functions belonging to \(L^2(\Omega)\) that admit a  derivative with respect to $x$  satisfying condition \eqref{condition2}.
  \end{remark} 
  
  \begin{theorem}\label{Theorem4.6}
Let \( h(S(T),V (T ))=\iota(S(T))\times  \kappa(V (T )) \in L^2(\Omega) \) and \( S(t) \) be a process of the form \eqref{stt}.
Assume that 
\begin{enumerate}
   \item there exists an integrable random variable \(H\) such that
\begin{equation}\label{condition2}
    \big|\frac{\partial}{\partial x}  \iota(S(t))\times\kappa(V(t))\big| \le H \quad \text{a.s. for all } t \in [0,T],
\end{equation}

\item the derivative of the process $S$ satisfies   equation \ref{FIVA}.
\end{enumerate}
Define
\[
\Lambda =\left \{ \theta \in L^{2}(\left [ 0 ,T\right ]) | \int_{0}^{T} \theta(t)dt=1 \right \}.
\]      
Then, for every $\theta$ $\in$  $\Lambda$ it holds, 
\begin{align}
\Delta_{S}  
  & =  \frac{1}{\sqrt{1 - \rho^2} }\mathbb{E}\big[h(S(T),V(T))  \times\int_{0}^{T}\theta(t) \sigma _{c}^{-1}(S^{c} (t))Z_{1}(t)dW_1(t) \big].
\end{align}
  \end{theorem}
  \begin{proof}
     First, we prove the theorem for  \( \iota, \kappa\in C_K^\infty(\mathbb{R}) \). In this case  $\iota (S(T))\times  \kappa(V (T )) \in L^2(\Omega).$ 
  Using the assumption \eqref{condition2}  in Theorem \ref{Theorem4.6}, we get 
 \begin{align}
\Delta_{S }  & =\frac{\partial }{\partial x}\mathbb{E}\Big[ h(S(T),V(T))\Big]\nonumber\\
        & =\frac{\partial }{\partial x}\mathbb{E}\Big[\iota(S(T))\times \kappa(V(T))\Big]\nonumber\\
        & =\mathbb{E}\Big[ \frac{\partial S(T)}{\partial x} \times \iota'(S(T))\times  \kappa(V (T) )\Big]\nonumber\\
        & =  \mathbb{E}\Big[ \frac{\partial S(T)}{\partial S^c(T)}\times \frac{\partial S^{c}(T)}{\partial x} \times \iota'(S(T))\times  \kappa (V (T ))\Big]\nonumber\\
        & = \mathbb{E}\Big[\frac{\partial S(T)}{\partial S^{c}(T)}Z_{1}(T) \iota'(S(T))\times \kappa(V (T ) )\Big].\nonumber
        \end{align}
By the chain rule, the   Malliavin derivative of  $S(T)$ is given by
\begin{align}\label{m}
     D_{t,0}(S((T)))&=D_{t,0} \, \psi_{1}(S^{c}(T),S^{d}(T))\nonumber\\
     &=\frac{\partial \psi_{1}}{\partial x}(S^{c}(T),S^{d}(T))\times D_{t}^{W_{1}}(S^{c}(T))\nonumber\\
     &=\frac{\partial S(T)}{\partial S^{c}(T)} \times D_{t}^{W_1}(S^{c}(T)).
\end{align} 
 where $ D_{t}^{W_{1}}$ is the Malliavin derivative with respect to the Brownian motion $W_{1}$.   
By the Lemma 7.5.5 in  \cite{nualart2006malliavin}, we have  
\begin{equation}\label{mallvien 2}
    D_{t}^{W_1}(S^{c}(T))=Z_{1}(T)\times Z_{1}^{-1}( t)\sqrt{1 - \rho^2} \sigma _{c}(S^{c} (t) ) \mathbb{1}_{\left \{ t \leqslant T  \right \}}.
\end{equation}
 Replacing \eqref{mallvien 2}in \eqref{m}, we get
 \begin{equation*}   
    D_{t,0}(S((T)))=  \frac{\partial S(T)}{\partial S^{c}(T)} \times Z_{1}(T)\times Z_{1}^{-1}( t)\sqrt{1 - \rho^2} \sigma _{c}(S^{c} (t) ) \mathbb{1}_{\left \{ t \leqslant T  \right \}},
\end{equation*}
then
\begin{equation}\label{10}
  \frac{\partial S(T)}{\partial S^{c}(T)} \times Z_{1}(T) = \frac{1}{\sqrt{1 - \rho^2} } D_{t,0}(S((T))) \times Z_{1} ( t) \times \sigma _{c}^{-1}(S^{c} (t) ). 
\end{equation}
Multiplying \eqref{10} by  $\theta(t)$ then integrating over  the  interval  $[0,T]$,
we get 
 \begin{equation}\label{ed20.} 
 \frac{\partial S(T)} {\partial S^{c}(T)}  \times Z_{1}(T)  = \frac{1}{\sqrt{1 - \rho^2} }\int_{0}^{T}  D_{t,0}(S((T))) \times Z_{1} ( t) \times \sigma _{c}^{-1}(S^{c} (t) ) \theta(t)dt.
 \end{equation}
Replacing \eqref{ed20.} in the expression of delta, we get :
\begin{flalign*}
\Delta   & =  \frac{1}{\sqrt{1 - \rho^2} }\mathbb{E}\Big[ \iota'(S(T))\times \kappa(V (T ) )\int_{0}^{T}  D_{t,0}(S(T)) \times Z_{1} ( t) \times \sigma _{c}^{-1}(S^{c} (t) )\theta(t) dt\Big]&\\
   & = \frac{1}{\sqrt{1 - \rho^2} } \mathbb{E}\Big[   \int_{0}^{T}  D_{t,0}(S(T))\iota'(S(T))\times  \kappa(V (T ) ) \times Z_{1} ( t) \times \sigma _{c}^{-1}(S^{c} (t) )\theta(t) dt\Big] &\\
   & = \frac{1}{\sqrt{1 - \rho^2} } \mathbb{E}\Big[ \int_{0}^{T}  D_{t,0} (  \iota(S(T))\times  \kappa(V (T ) ))  \times Z_{1} ( t) \times \sigma _{c}^{-1}(S^{c} (t) )\theta(t) dt\Big] &\\
   & =  \frac{1}{\sqrt{1 - \rho^2} }\mathbb{E}\Big[  \iota(S(T))\times  \kappa(V (T ) ) \int_{0}^{T}   Z_{1} ( t) \times \sigma _{c}^{-1}(S^{c} (t) )\theta(t) dW_{1}(t)\Big]&\\
   & =\frac{1}{\sqrt{1 - \rho^2} }\mathbb{E}\Big[  h(S(T),V(T)) \int_{0}^{T}   Z_{1} ( t) \times \sigma _{c}^{-1}(S^{c} (t) )\theta(t) dW_{1}(t)\Big].
\end{flalign*}
We can extend $\iota$ and $\kappa$ beyond $C_K^{\infty}$ to  the case of $\iota(S(T))\times  \kappa(V (T )) \in L^2$ using the same technique as in \cite{benth2010robustness,fournie2001applications}.
  \end{proof}    
  
\subsection{Delta with respect to the volume}
Let  $V(t)$ be the volume process.
\begin{equation*}
    \begin{cases}
        dV (t) = a(V(t)) dt + b(V(t)) dW_2(t)  + \int_{\mathbb{R}^{*}} c(V(t), z) d\widetilde{N}(dt, dz) \\ 
        V (0) = v 
    \end{cases}
\end{equation*}
 and $S(t)$   be the price process.
\begin{equation*}
    \begin{cases}
        dS(t) = \alpha(S(t)) dt + \sigma(S(t)) dW (t) + \int_{\mathbb{R}^{*}} \varphi(S(t), z) \widetilde{N}(dt, dz) \\ 
        S(0) = x.
    \end{cases}
\end{equation*}
Here,   $W(t)$   and $W_2(t)$ are correlated by the relation 
\begin{equation}\label{rho1}
    dW_2(t)=\rho_1 dW(t)+ \sqrt{1-\rho_1^2}  d\tilde{W_1}(t)
\end{equation}
with $\tilde{W_1}(t)$ is independent  of both  $ W(t)$   and $ W_2(t)$.
The full system governing the dynamics of \( S(t) \) and \( V(t) \) is given by:
\begin{equation} \label{St }
     dS(t) = \alpha(S(t, w)) dt + \sigma(S(t, w)) dW (t) + \int_{\mathbb{R}^{*}} \varphi(S(t, w), z) \widetilde{N}(dt, dz), 
     \end{equation}
     and
    \begin{align} \label{vt}
        dV(t) &= a(V(t, \omega))\, dt + \rho_1\, b(V(t, \omega))\, dW(t)  + \sqrt{1 - \rho_1^2}\, b(V(t, \omega))\, d\widetilde{W}_1(t) \nonumber\\& + \int_{\mathbb{R}^*} c(V(t, \omega), z)\, \widetilde{N}(dt, dz).
    \end{align} 

By analogy to $S$, we repeat the same techniques for $V$. Suppose that there exists a continuous and differentiable function   $\psi_{2} $ with  bounded derivatives in the first variable  such that 
\begin{equation*} 
\begin{cases}
    V(t) =\psi_{2}(V^{ c}(t), V^{d}(t))\\ 
V(0)=v. 
\end{cases}
\end{equation*}
Here $ V^{ c} $ \ satisfies the stochastic differential equation,
\begin{equation*}  
\begin{cases}
    dV^{c}(t)=a_c(V^c(t, \omega))\, dt + \rho_1\, b_c(V^c(t, \omega))\, dW(t)  + \sqrt{1 - \rho_1^2}\, b_c(V^c(t, \omega))\, d\widetilde{W}_1(t) \\ 
V^{c}(0)=v.
\end{cases}
\end{equation*}
\begin{remark}
    Since $a_c$ and $b_c$ are $C^1$ and $V^c$ is  differentiable with respect to $v$ , the compositions $a_c(V^c(\cdot))$ and $b_c(V^c(\cdot))$ are almost everywhere differentiable, and by the chain rule
\begin{equation}\label{4.7}  
\frac{\partial}{\partial v} a_c(V^c (v)) = a_c'(V^c (v)) \cdot \frac{\partial V^c (v)}{\partial v}, 
\qquad 
\frac{\partial}{\partial v}  b_c(V^c (v)) = b_c'(V^c (v)) \cdot \frac{\partial V^c(s)}{\partial v}.
\end{equation}
So, the integrals of derivatives are well defined.
By \cite{ikeda2014stochastic} , we will apply the derivative with respect to $x$:  
 \begin{align} \label{FIVA}
     \frac{\partial V^{c}(t)}{\partial v}&=1+\int_{0}^{s}a _{c}'(V^c  (s) )\times \frac{\partial V^{c}(s)}{\partial v} ds   + \rho_1\int_{0}^{t} b'(V^c  (s) )\times \frac{\partial V^{c}(s)}{\partial v} dW(s)  \\
     &\qquad+  \sqrt{1 - \rho_1^2} \int_{0}^{t} b'(V  ^c(s) )\times \frac{\partial V^{c}(s)}{\partial v} d\tilde{W_1} (s).\nonumber
 \end{align}
\end{remark}

We associate to the process $V^{c}$, a process $Z_{2}$ such that,
\begin{equation*}
    Z_{2}(t)=\frac{\partial V^{c}(t)}{\partial v}=1+\int_{0}^{s}a _{c}'(V^{c}(s) )\times \frac{\partial V^{c}(s)}{\partial v} ds   + \int_{0}^{t} b  _{c}'(V ^{c}(t) )\times \frac{\partial V^{c}(s)}{\partial v} dW_{2}(s). 
\end{equation*}
The process $Z_{2}$ is called the first variation process for  $V^{c}$.    
\begin{theorem}
Let \( h(S(T),V (T ))=\iota(S(T))\times  \kappa(V (T )) \in L^2(\Omega) \) and \( V(t) \) be a process of the form \eqref{vt}.
Assume that 
\begin{enumerate}
   \item there exists an integrable random variable \(H_2\) such that
\[
\big|  \iota(S(t))\times\frac{\partial}{\partial v} \kappa(V(t))\big| \le H_2 \quad \text{a.s. for all } t \in [0,T], 
\] 
\item the derivative of the process $V$ satisfies  equation \ref{FIVA}.
\end{enumerate}
We consider the following set   
\begin{equation*}
    \Lambda =\left \{ \theta\in L^{2}(\left [ 0 ,T\right ]) | \int_{0}^{T} \theta(t)dt=1 \right \}.
\end{equation*}\label{18}

 Then for $\theta\in\Lambda$ , the delta with respect to $V$ is given by :
\begin{align}
    \Delta_{V}  &= 
   \frac{1}{\sqrt{1 - \rho_1^2} } \mathbb{E}\Big[h(S(T),V(T))\int_{0}^{T}\theta(t) b _{c}^{-1}(V^{c}(t))Z_{2}(t)d\tilde{W_{1}}(t)\Big]\nonumber   
.\end{align}
\end{theorem}
\begin{remark}
At first sight, the modelling frameworks used in the density-based approach and in the Malliavin approach may appear different: in the former, the computation is performed directly at the level of the factor densities, whereas in the latter we impose explicit SDE dynamics for the factors. This is not a contradiction but rather a natural methodological choice. For the density method, it is advantageous to work with abstract factors whose joint density is assumed to exist, since this allows us to exploit the structural properties of the model and to derive the required expressions in a flexible way, independently of any specific stochastic dynamics. In contrast, Malliavin calculus operates on the level of the underlying noise---the Brownian motions and L\'{e}vy processes driving the system---and therefore requires a more explicit specification of the factor dynamics through SDEs.
For clarity of exposition and to avoid unnecessary technicalities, we restrict ourselves in the Malliavin section to a single driving factor for the price component and one for the volume component. This reduced setting is sufficient to illustrate the method and to make the core ideas transparent. Nevertheless, the Malliavin approach naturally extends to a genuinely multidimensional framework with several factors driven by multiple Brownian motions and/or L\'{e}vy processes. The generalization follows the same lines: one introduces the corresponding Malliavin derivatives with respect to each noise source and obtains sensitivity representations involving stochastic integrals against all driving processes. In fact, this structure already appears implicitly in model \eqref{CD}, where both $W_1$ and $W_2$ arise, the latter entering through approximation. A fully vector-valued formulation could be developed analogously, but we choose not to pursue it here in order to keep the presentation focused and accessible.
\end{remark}

\section{Examples}
We use Malliavin calculus to study Greeks for time-changed Brownian motion models, then for Lévy-driven models, and finally for pure-jump models where small jumps are approximated by a suitably scaled Brownian motion. The efficiency of the Malliavin method lies in the fact that it does not rely on the existence or explicit form of the density; it is much more general than that.

In the next example, we consider the same model as in Example 3.1 but with the Malliavin method, and we will demonstrate numerically that the deltas of the two methods coincide.
\subsection{A correlated one-factor price and one-factor volume model}
We consider  the example  of Subsection \ref{example 1  1P 1V}.   Here, we re-estimate the delta of the model by applying the Malliavin calculus technique. Recall the Ornstein-Uhlenbeck processes
\[
\begin{aligned}
dX(t) &= -\alpha_X X(t)\,dt + \beta_X\,dW(t), \qquad \\[0.2cm]
dY(t) &= -\alpha_Y Y(t)\,dt + \eta\,d\widetilde W(t)\qquad  
\end{aligned}
\]
where the Brownian motions satisfy
\[
dW(t)=\rho\,d\widetilde W(t)+\sqrt{1-\rho^2}\,d\widetilde W_1(t),
\]
with $\widetilde W_1(t)$ independent of $\widetilde W(t)$.
We define
\begin{equation}\label{111
}
    S(t)=S(0) e^{X(t)}, \qquad V(t)=V(0) e^{Y(t)}
\end{equation}
where \( S(0)=x >0\) and \( V(0)=v\).

By applying Itô’s formula to $S(t)=S(0)e^{X(t)}$ gives
\[
\begin{aligned}
dS(t)
&= \big(-\alpha_X\ln((\frac{S(t)}{x})+\tfrac12\beta_X^2\big)S(t)\,dt
    +\beta_X S(t)\,dW(t)\\[4pt]
&= \big(-\alpha_X\ln(\frac{S(t)}{x}) +\tfrac12\beta_X^2\big)S(t)\,dt
    +\beta_X\rho\, S(t)\,d\widetilde W(t)
    +\beta_X\sqrt{1-\rho^2}\,S(t)\,d\widetilde W_1(t).
\end{aligned}
\]
Hence, the drift and diffusion are
\[
\alpha_c(S(t))=\left(-\alpha_X\ln (\frac{S(t)}{x})+\tfrac12\beta_X^2\right)S(t),
\qquad
\sigma_c(S(t))=\beta_X S(t).
\]
If the initial condition is $S(0)=x$, then the flow derivative with respect to $x$ is
\[
Z_1(t)=\frac{\partial S(t)}{\partial x}=\frac{S(t)}{x}.
\]
  Using the Proposition \ref{Theorem4.6}
  for \( h(S(T), V(T))= \iota(S(T))\times  \kappa(V (T )) \in L^2(\Omega) \) and satisfying condition \eqref{condition2}, we obtain
\[
\Delta_S
=\frac{1}{S(0) \beta_X\sqrt{1-\rho^2}}\;
\mathbb{E}\!\left[
h(S(T), V(T))
\int_0^T  \theta(t)\,d\widetilde W_1(t)
\right].
\]
By choosing $\theta(t)=\frac{1}{T}$, we get \[
\Delta_S
=\frac{1}{TS(0) \beta_X\sqrt{1-\rho^2}}\;
\mathbb{E}\!\left[
h(S(T), V(T))\widetilde W_1(T)
\right].
\]
\begin{remark}
  Both the density method and Malliavin calculus approach provide powerful theoretical expressions for calculating the delta $\Delta_S$ using Monte Carlo simulations. As an illustration, we performed a simple Monte Carlo experiment for testing the two approaches. Both methods are run using the model parameters $S_0=100$, $V_0=0.50$, $X_0=0$, $Y_0=0$, $\alpha_X=2$, $\beta_X=3$, $\alpha_Y=0.15$, $\eta=0.25$ and $\rho=0.45$. We considered the payoff is defined as the product of a price call and a volume put, expressed as \( \max(S_T - K, 0) \times \max(L - V_T, 0)\), with exercise time   $T=1.0$ (1 year), price strike $K=100$ and put strike $L=31.5$. The interest rate was set to $r=0.025$ (2.5$\%$ annually). With  $n_{\text{paths}}=20{,}000$ samples paths in the Monte Carlo simulation, we obtained the values 0.126 for the Malliavin method and 0.138 for the density (using the same starting seed in the random generation). 
  Note that the Malliavin implementation may introduce a minor discretization error because it simulates over finite time steps, whereas the density method uses the analytical terminal distribution. 
\end{remark}

\subsection{ Jump diffusion driven by a compound Poisson process }
Consider a jump diffusion of the form 
\begin{equation}
\begin{cases}
    dS(t)=\alpha_S S(t) dt+ \sigma_S S(t) dW(t)+\gamma_S S(t)dL(t),\\
    S(0)=x,
\end{cases}
\end{equation}
where \(L(t)\) is a compound Poisson process independent of the Brownian motion \(W(t)\).
The compound Poisson process \(L(t)\)  can be represented as:
\begin{equation*}
L(t)=\int_0^T\int_{\mathbb{R}}zN(ds,dz),
\end{equation*}
where $N$ is a Poisson random measure on $[0,\infty)\times \mathbb{R}$ with intensity $\ell(dz)dt$ where $\ell$ is the L\' evy measure. Moreover,  

The jump diffusion process satisfies then the following SDE:
\begin{equation}\label{sdejump}
\begin{cases}
    dS(t)=\alpha_S S(t) dt+ \sigma_S S(t) dW(t)+\gamma_S\int_{\mathbb{R}}zS(t)N(dt,dz),\\
    S(0)=x,
 \end{cases}
\end{equation}
And let $V(t)$ The jump diffusion process satisfies then the following SDE:
\begin{equation}\label{sdejump2}
\begin{cases}
    dV(t)=\alpha_V V(t) dt+ \sigma_V V(t) dW_2(t)+\gamma_V\int_{\mathbb{R}}zV(t)N(dt,dz),\\
    V(0)=v,
 \end{cases}
\end{equation}
 
where   $ W(t)  $   and $ W_2(t) $ are correlated by the relation 
\begin{equation*}
    dW(t)=\rho dW_2(t)+ \sqrt{1-\rho^2}  dW_1(t)
\end{equation*}
with $W_1(t)$ is independent both  $W(t)$ and $W_2(t)$.
\begin{proposition}
Let $(S_t)_{t\in[0,T]}$ satisfy the SDE \eqref{sdejump}. Assume that the payoff
\[
h(S(T),V(T))
=
\iota(S(T))\,\kappa(V(T))
\]
belongs to \(L^2(\Omega)\) and that condition \eqref{condition2} holds. Then the delta with respect to the price is given by
\begin{equation}
\Delta_{S}
=
\frac{1}{\sqrt{1-\rho^2}}\,
\frac{1}{\sigma_S x T}\,
\mathbb{E}\!\left[
h(S(T),V(T))\,W_1(T)
\right].
\end{equation}
\end{proposition}
\begin{proof}
Using the It\^o formula with jump for the function $\ln$ we get the solution of \eqref{sdejump} given by
\begin{equation}
    S(t)=x\exp\{(\alpha_S-\frac{1}{2}\sigma_S^2)t+\sigma_S W_t+\int_0^t\int_{\mathbb{R}}\ln(1+\gamma_S z)N(dt,dz)\}.
\end{equation}
In this case $S^c(t)$ is given by 
\begin{equation*}
    S^c(t) = x\exp\{(\alpha_S-\frac{1}{2}\sigma_S^2)t+\sigma_S W(t)\} .
\end{equation*} 
 We have 
\begin{equation*}
    Z_{1}(t)= \frac{\partial S^{c}(t)}{\partial x} = \exp \{( \alpha _S t  -\frac{1}{2}\sigma_S^2) +   \sigma_S W(t)\} = \frac{S^c (t)}{x}.
\end{equation*}
 In this example, $\sigma_c(x)=\sigma_S x$ and therefore  $\sigma^{-1}_c(x)=\frac{1}{\sigma_S x}$. By choosing $ \theta(t) = \frac{1}{T}$, we get 
 \begin{align*}
 \Delta_{S} &=\frac{1}{\sqrt{1 - \rho^2} }\mathbb{E}\Big[ h(S(T),V(T))  \times\int_{0}^{T}\theta(t) \sigma _{c}^{-1}(S^{c}(t))Z_{1}(t)dW_1(t)\Big]\\
 &=\frac{1}{\sqrt{1 - \rho^2} } \mathbb{E}\Big[h(S(T),V(T))  \times\int_{0}^{T} \frac{1}{\sigma_S.x.T}dW_1(t)\Big]\\
        &= \frac{1}{\sigma_S x T\sqrt{1 - \rho^2}  }\mathbb{E}\Big[h(S(T),V(T))     W_1(T) \Big]. 
        \end{align*}
         
    \end{proof}
   
\begin{remark} Consider the particular case where the payoff function $h$ is given by:
 \begin{equation}\label{pay}
      h(S(T), V(T)) =\iota(S((T)) \, \kappa(V(T)) = \max(K - S(T), 0) \times \max(V(T) - L, 0),  
 \end{equation}
 and assume that 
the Lévy measure \(\ell\) satisfies
\begin{equation}\label{moment4}
    \int_{\mathbb R} |z|^k \,\ell(dz)<\infty, \text{ for all } k.
\end{equation}
We verify that $\iota(S((T)) \, \kappa(V(T))$ in $L^2$.
Since \[
\iota(S(T))=\max(K-S(T),0),
\]
hence
\[
0 \le \iota(S(T)) \le |K-S(T)| \le K + |S(T)|.
\]
Since $$\kappa(V(T))= \max(V(T) - L, 0)$$
we similarly obtain
\[
0 \le \kappa(V(T)) \le |V(T) - L| \le L + |V(T)|.
\]
Hence
\[
|\iota(S(T))\kappa(V(T))|
\le (K + |S(T)|)(L + |V(T)|).
\]
Therefore,
\[
\mathbb{E}\big[|\iota(S(T))\kappa(V(T))|^2\big]
\le C\left(1 + \mathbb{E}[|S(T)|^2] + \mathbb{E}[|V(T)|^2] + \mathbb{E}[|S(T)|^2|V(T)|^2]\right),
\]
for some constant \(C>0\).
For the mixed term, we apply the Cauchy-Schwarz inequality:
\[
\mathbb{E}\big[|S(T)|^2 |V(T)|^2\big]
\le \left(\mathbb{E}[|S(T)|^4]\right)^{1/2}
\left(\mathbb{E}[|V(T)|^4]\right)^{1/2}.
\]
First,
\[
S(T)^2
=
x^2
\exp\left\{
2\left(\alpha_S-\frac12\sigma_S^2\right)T
+2\sigma_S W_T
+2\int_0^T\int_{\mathbb R}\ln(1+\gamma_S z)\,N(ds,dz)
\right\}.
\]

Using the independence of the Brownian and jump parts,
\[
\mathbb E[S(T)^2]
=
x^2
e^{2(\alpha_S-\frac12\sigma_S^2)T}
\mathbb E[e^{2\sigma_S W_T}]
\,
\mathbb E\left[
e^{2\int_0^T\int_{\mathbb R}\ln(1+\gamma_S z)\,N(ds,dz)}
\right].\]
Since
\[
\mathbb E[e^{2\sigma_S W_T}]
=
e^{2\sigma_S^2 T},
\]
and
\[\mathbb E\left[
e^{\int_0^T\int_{\mathbb R} 2\ln(1+\gamma_S z)\,N(ds,dz)}
\right]
=
\exp\left(
T\int_{\mathbb R}(e^{2\ln(1+\gamma_S z)}-1)\,\ell(dz)
\right),
\]

hence,
\[
\mathbb E[S(T)^2]
=
x^2
\exp\left(
2\alpha_S T
+
\sigma_S^2 T
+
T\int_{\mathbb R}\big((1+\gamma_S z)^2-1\big)\ell(dz)
\right).
\]
The condition \eqref{moment4} guaranties that $\mathbb E[S(T)^2]<\infty$.
Similarly,
\[
S(T)^4
=
x^4
\exp\left\{
4\left(\alpha_S-\frac12\sigma_S^2\right)T
+4\sigma_S W_T
+4\int_0^T\int_{\mathbb R}\ln(1+\gamma_X z)\,N(ds,dz)
\right\}.
\]
Using
\[
\mathbb E[e^{4\sigma_S W_T}]
=
e^{8\sigma_S^2 T},
\]
and
\[
e^{4\ln(1+\gamma_X z)}
=
(1+\gamma_X z)^4,
\]
we get
\[
\mathbb E[S(T)^4]
=
x^4
\exp\left(
4\alpha_S T
+
6\sigma_S^2 T
+
T\int_{\mathbb R}\big((1+\gamma_X z)^4-1\big)\ell(dz)
\right).
\]
To guarantee
\[
\int_{\mathbb R}\big((1+\gamma_S z)^4-1\big)\ell(dz)<\infty,
\]
one needs a suitable moment condition on the Lévy measure \(\ell\).

Expanding the polynomial gives
\[
(1+\gamma_S z)^4-1
=
4\gamma_S z
+6\gamma_S^2 z^2
+4\gamma_S^3 z^3
+\gamma_S^4 z^4.
\]
Hence, finiteness follows if
\[
\int_{\mathbb R}
\left(
|z|
+|z|^2
+|z|^3
+|z|^4
\right)\ell(dz)<\infty.
\]
The condition \eqref{moment4} guarantees that $\mathbb E[S(T)^4]<\infty$.
The same arguments can be applied to prove that $\mathbb E[V(T)^2]<\infty$ and $\mathbb E[V(T)^4]<\infty$.
Now we verify condition \eqref{condition2} of Theorem \ref{Theorem4.6}.
We have
\[
h(S(T),V(T))
=
\iota(S(T))\,\kappa(V(T)).
\]

Since only \(S(T)\) depends on the initial condition \(x\),
\[
\frac{\partial}{\partial x}
\big(\iota(S(T))\kappa(V(T))\big)
=
\kappa(V(T))
\frac{\partial}{\partial x}\iota(S(T)).
\]
Now,
\[
\iota(y)=\max(K-y,0),
\]
and hence
\[
\iota'(y)
=
-\mathbf 1_{\{y<K\}}.
\]
Therefore,
\[
\frac{\partial}{\partial x}\iota(S(T))
=
-\mathbf 1_{\{S(T)<K\}}
\frac{\partial S(T)}{\partial x}.
\]
Since
\[
S(T)
=
x\exp\left\{
\left(\alpha_S-\frac12\sigma_S^2\right)T
+\sigma_S W_T
+\int_0^T\int_{\mathbb R}\ln(1+\gamma_S z)\,N(ds,dz)
\right\},
\]
we obtain
\[
\frac{\partial S(T)}{\partial x}
=
\exp\left\{
\left(\alpha_S-\frac12\sigma_S^2\right)T
+\sigma_S W_T
+\int_0^T\int_{\mathbb R}\ln(1+\gamma_S z)\,N(ds,dz)
\right\}
=
\frac{S(T)}{x}.
\]
Hence,
\[
\left|
\frac{\partial}{\partial x}
\big(\iota(S(T))\kappa(V(T))\big)
\right|
=
\mathbf 1_{\{S(T)<K\}}
\frac{S(T)}{x}
|\kappa(V(T))|.
\]
Using
\[
\mathbf 1_{\{S(T)<K\}}S(T)\le K,
\]
we obtain
\[
\left|
\frac{\partial}{\partial x}
\big(\iota(S(T))\kappa(V(T))\big)
\right|
\le
\frac{K}{x}|\kappa(V(T))|.
\]
Finally, since
\[
|\kappa(V(T))|
=
|\max(V(T)-L,0)|
\le
L+|V(T)|,
\]
we get
\[
\left|
\frac{\partial}{\partial x}
\big(\iota(S(T))\kappa(V(T))\big)
\right|
\le
\frac{K}{x}(L+|V(T)|).
\]
Thus, one may choose
\[
H=\frac{K}{x}(L+|V(T)|).
\]
Since \(V(T)\) admits finite moments, \(H\) is integrable.
Since the conditions of Theorem \ref{Theorem4.6} are verified,
then the delta becomes 
 \begin{equation*}
\Delta_S= \frac{1}{\sigma_S x T\sqrt{1 - \rho^2}  } \, \mathbb{E}\Bigg[ h(S(T), V(T)) \cdot W_1(T) \Bigg],
\end{equation*}
where  $h(S(T), V(T))$ is given by \eqref{pay}.

\end{remark} 
\subsection{ Time changed  Brownian motion  with IG process }
In this example, we show that the time changed Brownian motion by an inverse Gaussian process (IG) is a normal inverse Gaussian process (NIG). This is already done in\cite{rydberg1997normal,asmussen2001approximations}, but here  we give the proof to highlight the  choice of the parameters of the NIG process.
 \begin{lemma}
     Let \( W(t) \) a Brownian motion and \( N(t) \)  an independent inverse Gaussian process with parameters \( (\mu, \nu)\), then the process \( W(N(t)) \) follows a normal inverse Gaussian (NIG) distribution with parameters \( (\frac{\sqrt{\nu}}{\mu}, 0,  1,   0) \).
 \end{lemma}
 \begin{proof}
We aim to find the characteristic function of the process \( W(N(t)) \), where \( W(t) \) is a Brownian motion and \( N(t) \) is an independent inverse Gaussian process. 
We can express the characteristic function of \( W(N(t)) \) by conditioning on \( N(t) \):
\[
\mathbb{E}[e^{iuW(N(t))}] = \mathbb{E}[\mathbb{E}[e^{iuW(N(t))} \mid N(t)]].
\]
Given \( N(t) = n \), we have \( W(N(t)) = W(n) \), where \( W(n) \sim \mathcal{N}(0, n) \). The characteristic function of \( W(n) \) is:
\[
\mathbb{E}[e^{iuW(n)}] = e^{-\frac{u^2}{2}n}.
\]
Thus, we have:
\[
\mathbb{E}[e^{iuW(N(t))} \mid N(t)] = e^{-\frac{u^2}{2}N(t)}.
\]
Now, we take the expectation of \( e^{-\frac{u^2}{2}N(t)} \) with respect to the distribution of \( N(t) \), where \( N(t) \sim \text{IG}(\mu, \nu) \) (inverse Gaussian distribution with mean \( \mu \) and shape parameter \( \nu\)). The moment generating function of \( N(t) \) is:
\[
\mathbb{E}[e^{-\frac{u^2}{2}N(t)}] = \exp\left(\frac{\nu}{\mu} \left( 1 - \sqrt{1 + \frac{u^2}{\nu\mu^2}} \right)\right).
\]
Thus, the characteristic function of \( W(N(t)) \) is:
\[
\mathbb{E}[e^{iuW(N(t))}] = \exp\left(\frac{\nu}{\mu} \left( 1 - \sqrt{1 + \frac{u^2}{\nu \mu^2}} \right)\right).
\]
This characteristic function reflects the combined effects of the Brownian motion and the inverse Gaussian time-change, resulting in a process with heavy tails and asymmetry.

Recall that for a random variable \( Y \sim \text{NIG}(\alpha, \beta, \delta, c) \), the characteristic function is given by:
\[
\phi_Y(u) =  
\exp(  ic u + \delta \left( \gamma - \sqrt{\alpha^2 - (\beta + iu)^2} \right))
 \]
where:
\begin{itemize}
    \item \( \alpha > 0 \): shape parameter,
    \item \( |\beta| < \alpha \): asymmetry parameter,
    \item \( \delta > 0 \): scale parameter,
    \item \( c \in \mathbb{R} \): location parameter.
    \item \( \gamma =\sqrt{\alpha^2 - \beta^2}  \).  
\end{itemize}
For a random variable \( Y \sim \text{NIG}\left(\frac{\sqrt\nu}{\mu}, 0, 1, 0\right) \), where:
\[
\alpha = \frac{\sqrt \nu}{\mu}, \quad \beta = 0, \quad \delta = 1, \quad  c = 0,
\]
the characteristic function simplifies to :
\[
\phi_Y(u) =\exp\left(\frac{\nu}{\mu} \left( 1 - \sqrt{1 + \frac{u^2}{\nu \mu^2}} \right)\right)=\phi_{W(N(t))}(u).
\].
\\
So, \( W(N(t)) \) follows a normal inverse Gaussian (NIG) distribution with parameters \\\( (\frac{\sqrt{\nu}}{\mu}, 0,  1,   0) \).
 \end{proof}
 \begin{remark}
     The NIG process $Y$ can be also expressed as a normal variance-mean mixture: $Y=c+\beta Z+ \sqrt{Z} X$, where $Z\sim IG(\delta, \gamma)$ and $X \sim N(0,1)$, for more details see \cite{barndorff1997normal}.
 \end{remark}
\begin{remark}\label{rem3}
The NIG process can be classified as a pure jump Lévy process due to the nature of its construction and the absence of a continuous martingale component in its Lévy-Khintchine formula.
 In the centered (mean-zero) case, the NIG process \( Y_t \) can be expressed as: 
\[
Y_t = \int_0^t \int_{\mathbb{R}} z \, \tilde{N}(ds, dz),
\]
where
\[
\tilde{N}(ds, dz) = N(ds, dz) - \ell(dz) \, ds
\]
is the compensated Poisson random measure and \( \ell(dz) \) is the Lévy measure given by 
\begin{equation}\label{numeasure}
    \ell(dz) = \frac{\sqrt{\nu}}{\pi\mu|z|}  K_{1}(\frac{\sqrt{\nu}|z|}{\mu}) \, dz,
\end{equation}
where $K_1$ is the modified Bessel function of the second kind.
\end{remark}
 Assume that the price satisfies the stochastic differential equation given by 
\begin{equation}\label{corees}
    dW(N(t))=dS(t) \text{ 
 where  }  S(0)=x.  
 \end{equation}
We compute the delta corresponding to the price process $( S(t))_{t\in[0,T]}$ in \eqref{corees} using two methods: the first one is by using the conditioning and the second one is by using Remark \ref{rem3} and approximating the small jumps by a scaled Brownian motion. \\

\subsection{Two dependent NIG process}
  By the same model   as in \eqref{corees}, we assume that the volume satisfies the  stochastic differential equation  given by:
\begin{equation}\label{corees1}
   dV(t)=dW_2(N(t))  \text{ 
 where  }  V(0)=v.  
 \end{equation}
 Now, we assume     $ W(t)  $   and $ W_2(t) $ are correlated by the relation 
\begin{equation*} 
    dW(t)=\rho dW_2(t)+ \sqrt{1-\rho^2}  dW_1(t)
\end{equation*}
with $W_1(t)$ is independent  of both  $ W(t)$   and $ W_2(t)$.

First we want to calculate the correlation between \( X(t) \) and \( V(t) \).
Since \( W_1 \) and \( W_2 \) are independent standard Brownian motions and independent of \( N(t) \), we have:
\begin{align*}
\mathbb{E}[N(t)] &= \mu \\
\mathrm{Var}(V(t)) &= \mathrm{Var}( v+W_2(N(t)))= \mathrm{Var}(W_2(N(t)))= \mathbb{E}[W_2(N(t))^2] -\mathbb{E}[W_2(N(t))]^2 \\ &=\mathbb{E}\big(\mathbb{E}[W_2(N(t))^2]|N(t))- \mathbb{E}\big[\mathbb{E}[W_2(N(t))|N(t)]\big]^2  \\
&=\mathbb{E}\big(\mathbb{E}[W_2(u)^2]|u=N(t))- \mathbb{E}\big[\mathbb{E}[W_2(u)|u=N(t)]\big]^2  \\
&=\mathbb{E}[N(t)]-0\\
&=\mu.
\end{align*}
By the same technique, we compute $\mathrm{Var}(S(t) )$. Then we have
\begin{align*}
\mathrm{Var}(S(t) ) &=\mathrm{Var}(x+W(N(t)) )=\mathrm{Var}(W(N(t)) )= \rho^2 \mathbb{E}[N(t)] + (1 - \rho^2) \mathbb{E}[N(t)] = \mu. 
\end{align*}
By applying the same method and using conditional expectation, we calculate the covariance between $X(t) $ and $V(t)$ to obtain
\begin{align*}
\mathrm{Cov}(S(t), V(t)) 
&= \mathrm{Cov}\left(x+\rho W_2(N(t)) + \sqrt{1 - \rho^2} \, W_1(N(t)) ,\; v+W_2(N(t)) \right) \\
&= \mathrm{Cov}\left(\rho W_2(N(t)) + \sqrt{1 - \rho^2} \, W_1(N(t)) ,\; W_2(N(t)) \right) \\
&= \rho \cdot \mathrm{Var}(W_2(N(t))) + \sqrt{1 - \rho^2} \cdot \underbrace{\mathrm{Cov}(W_1(N(t)), W_2(N(t)))}_{= 0} \\
&= \rho \cdot \mu. 
\end{align*}
So we have \[
\mathrm{Corr}(S(t), V(t)) = \frac{\rho \mu }{\sqrt{\mu \cdot \mu }} = \frac{\rho \mu }{\mu } = \rho.
\]
This gives us the following result for the delta when applying conditioning:
\begin{proposition}\label{Prop4.2}
Let \((S(t))_{t\in[0,T]}\) and \((V(t))_{t\in[0,T]}\) be the processes satisfying the stochastic differential equations \eqref{corees} and \eqref{corees1}, respectively. Assume that
\[
h(S(T),V(T))
=
\iota(S(T))\,\kappa(V(T))
\in L^2(\Omega),
\]
and that there exists an integrable random variable \(H\) such that
\[
\left|
\frac{\partial}{\partial x}
\big(\iota(S(T))\,\kappa(V(T))\big)
\right|
\le H
\qquad \text{a.s.}
\]
Then the delta corresponding to the process \((S(t))_{t\in[0,T]}\) is given by
\[
\Delta_S
= \frac{1}{\sqrt{1 - \rho^2} }
\mathbb{E}\!\left[
h(S(T),V(T))
\frac{W_1((N(T))-\tau}{N(T)}
\right],
\] where $\tau=W_1(0)$.
\end{proposition}

\begin{proof}
Using the conditional expectation with respect to the IG process $N(T)$ and the independence of $W(T)$ and $N(T)$ we get:

    \begin{flalign*}
\Delta_{W(N ) }  & =\frac{\partial }{\partial x}\mathbb{E}\Big[ h(S(T),V( T))\Big]&\\
        &=\frac{\partial }{\partial x}\mathbb{E}\Big[\mathbb{E}\Big[h(x+W (N(T)),v+W_2( N(T))|N(T)\Big]\Big]&\\
          &=\mathbb{E}\Big[\frac{\partial }{\partial x}\mathbb{E}\Big[h(x+W(z),v+W_2(z))\Big]|_{z=N(T)}\Big]. 
        \end{flalign*}
Note that delta with respect to $W(z)$ is given by 
\begin{equation}
    \Delta_{W(z)} =\frac{\partial}{\partial x}\mathbb{E}\Big[ h(x+W(z),v+W_2(z))\Big].
\end{equation}
We consider
 $S(t)=x+W(t)$ where $dS(t)=dW(t)$ and $S(0)=x$  then  $S^c(t)=S(t)$, $\alpha_c=\alpha=0$ and $\sigma=\sigma_{c}=1$. We choose $\theta(t)= \frac{1}{z}$ with $T=z$. We use the expression in \eqref{FIVA1} to get the first variation $S^c(t)$, $Z_{1}(t)= \frac{\partial S^{c}(t)}{\partial x}=1$.
According to Theorem \ref{Theorem4.6}, we compute the delta with the Malliavin method to get,

\begin{flalign*}
\
\Delta_{W(z)}  & = \frac{1}{\sqrt{1 - \rho^2} }\mathbb{E}\Big[ h(x+W(z),v+W_2(z))\times\int_0^z \frac{1}{z} dW_1(t)\Big]&\\
        & = \frac{1}{\sqrt{1 - \rho^2} }\mathbb{E}\Big[ h(x+W(z),v+W_2(z))\frac{W_1(z)-W_1(0)}{z}\Big]&\\
         & = \frac{1}{\sqrt{1 - \rho^2} }\mathbb{E}\Big[ h(x+W(z),v+W_2(z))\frac{W_1(z)-\tau}{z}\Big]
        \end{flalign*}
where $\tau=W_1(0)$.
Now the delta of $S(T)$ is given by:

\begin{align*}
    \Delta_{S } &=\frac{1}{\sqrt{1 - \rho^2} }\mathbb{E}\Big[
    h(x+W(N(T)), v+W_2(N(T)))\frac{W_1((N(T))-\tau}{N(T)}\Big]\\
    &=\frac{1}{\sqrt{1 - \rho^2} }\mathbb{E}\Big[
    h(S(T),V(T))\frac{W_1((N(T))-\tau}{N(T)}\Big].
\end{align*}
    \end{proof}

In a second, alternative, approach we use the integral representation of the NIG process as a pure jump process and exploit the fact that we can approximate the small jumps by a scaled Brownian motion.
Let  $S(t) =x+W(N(t)) =x+\int_0^t\int_\mathbb{R}z \tilde{N}(ds,dz)=x+\int_0^t\int_\mathbb{R}N(ds,dz) -\int_0^t\int_\mathbb{R}\ell(dz)dt,$
where $x=S(0)$ and $\tilde{N}(ds, dz) = N(ds, dz) - \ell(dz) \, ds$ is the compensated Poisson random measure and \( \ell(dz) \) is the Lévy measure given by \eqref{numeasure}.

 
\[
\text{For all } 0 < \varepsilon \leq 1, \quad
dS(t) = \int_{0 < \left| z \right| \leq \varepsilon} z \tilde{N}(dt, dz) 
+ \int_{\left| z \right| > \varepsilon} z \tilde{N}(dt, dz).
\]
\\
 In various applications involving statistical and numerical methods, it is often
 useful to approximate the small jumps by a scaled Brownian motion. This approximation was advocated in Rydberg\cite{rydberg1997normal} as a way to simulate the path of a Lévy 
 process with NIG distributed increments, and later studied in detail by Asmussen and Rosinski \cite{asmussen2001approximations}.
 
 We introduce the following notation for the variation of the Lévy process \( L(t) \) close to the origin  
\[
\sigma^2(\varepsilon) := \int_{0 < |z| \leq \varepsilon} z^2 \, \ell(dz), \quad 0 < \varepsilon \leq 1.
\]
Then, by approximating the small jumps by a scaled Brownian motion $\sigma(\varepsilon)B(t)$ (and slightly abusing the notation),  
\begin{align}\label{scaledBM}
 dS(t) &= \sigma (\varepsilon)dB (t)
+ \int_{\left| z \right| > \varepsilon} z \tilde{N}(dt, dz) \nonumber\\
&=  \sigma (\varepsilon)dB (t)
+ \int_{\left| z \right| > \varepsilon} z    N (dt, dz) -\int_{\left| z \right| > \varepsilon}z \ell(dz)dt.
\end{align}
Now we introduce the process $V(t)$   such that :\\
\begin{align}\label{scaledBM2}
 dV(t) &= \sigma_V (\varepsilon)dB_1 (t)
+ \int_{\left| z \right| > \varepsilon} z \tilde{N}(dt, dz) \nonumber\\
&=  \sigma_V (\varepsilon)dB_1 (t)
+ \int_{\left| z \right| > \varepsilon} z    N (dt, dz) - \int_{\left| z \right| > \varepsilon}z \ell(dz)dt.
\end{align}
We assume     $ B(t)  $   and $ B_1(t) $ are correlated by the relation 
\begin{equation*} 
    dB(t)=\rho dB_1(t)+ \sqrt{1-\rho^2}  dB_2(t),
    \end{equation*}
with $B_2(t)$ is independent  of both  $ B(t)$   and $ B_1(t)$
 \begin{proposition}\label{Prop4.3}
 Let \((S(t))_{t\in[0,T]}\) and \((V(t))_{t\in[0,T]}\) be the processes satisfying \eqref{scaledBM} and \eqref{scaledBM2}, respectively. Assume that
\[
h(S(T),V(T))
=
\iota(S(T))\,\kappa(V(T))
\in L^2(\Omega),
\]
and that there exists an integrable random variable \(H\) such that
\[
\left|
\frac{\partial}{\partial x}
\big(\iota(S(T))\,\kappa(V(T))\big)
\right|
\le H.
\]
Then the delta corresponding to the process \(S\) defined by \eqref{scaledBM} is given by
    \begin{equation}
    \Delta_{S }=\frac{1}{\sqrt{1 - \rho^2} }\mathbb{E}\Big[ h(S (T),V(T))\frac{B(T)-y}{ T\sigma(\varepsilon)}\Big],    
\end{equation}
where $y= B(0).$ 
\end{proposition}
\begin{proof}
    Using the equations \eqref{scaledBM}and \eqref{scaledBM2}, one can see that $S^c(t), t\in[0,T]$ verifies the following SDE: 
\begin{equation}
    \begin{cases}
dS^{c}(t)=\sigma (\varepsilon)dB (t) -\int_{\left| z \right| > \varepsilon}z \ell(dz)dt\\ 
S^{c}(0)=x
\end{cases}
\end{equation} 
and
\begin{equation}
    \begin{cases}
dV^{c}(t)=\sigma_V (\varepsilon)dB_1(t) -\int_{\left| z \right| > \varepsilon}z \ell(dz)dt\\ 
V^{c}(0)=v.
\end{cases}
\end{equation}
Then 
$\alpha_c(S^{c})=-\int_{\left| z \right| > \varepsilon}z \ell(dz) , \sigma_c(S^{c})=\sigma(\varepsilon) , Z_1(t)=1$ and we choose $\theta(t)=\frac{1}{T}$. Using the Malliavin method in Theorem \ref{Theorem4.6}, the delta is given by :
\begin{equation}
    \Delta_{S}= \frac{1}{\sqrt{1 - \rho^2} }\mathbb{E}\Big[ h(S(T),V( T))\int_{0}^{T}\frac{1}{T\sigma(\varepsilon)}dB(t)\Big]=\frac{1}{\sqrt{1 - \rho^2} }\mathbb{E}\Big[ h(S(T),V( T))\frac{B(T)-y}{ T\sigma(\varepsilon)}\Big],    
\end{equation}
where $y= B(0).$ 
\end{proof}

\subsection{Empirical Example}
In this example, we estimate the parameter \(\Delta_{S}\) using two computational methods: Monte Carlo simulations and numerical integration, incorporating a normal-inverse-Gaussian (NIG) process in both methods presented in Proposition \ref{Prop4.2} and Proposition \ref{Prop4.3}.
\subsubsection{First Method: Monte Carlo simulation for a Normal Inverse Gaussian (NIG) process} Consider the following payoff function associated with Theorem \ref{Prop4.2}:
\begin{equation}
    h(S(T), V(T)) = \iota(S(T)) \kappa (V(T))=\max(K - S(T), 0) \times \max(V(T) - L, 0).
\end{equation}
To prove that $h(S(T), V(T))\in L^2$, 
we verify that $E[ h(S(T), V(T))
^2] < \infty$.
By the Cauchy-Schwarz inequality:
\[ E[\iota(S(T))^2 \kappa (V(T))^2] \leq \sqrt{E[\iota(S(T))^4] E[\kappa (V(T))^4]} \]
To show that $\mathbb{E}[\iota(S(T))^4] < \infty$, we evaluate the payoff function $\iota(\cdot)$ based on its position relative to the strike $K$.
\vspace{1em}\\

To show that $\mathbb{E}[\iota(S(T))^4] < \infty$, we have  that
\[
\iota(S(T))=\max(K-S(T),0).
\]
Hence
\[
0 \le \iota(S(T)) \le |K-S(T)| \le K + |S(T)|.
\]
Therefore
\[
\iota(S(T))^4 \le (K+|S(T)|)^4 \le C(1+|S(T)|^4)
\]
for some constant $C>0$. Consequently,
\[
\mathbb{E}[\iota(S(T))^4]
\le C\left(1+\mathbb{E}[|S(T)|^4]\right).
\]
If $\mathbb{E}[|S(T)|^4]<\infty$, it follows that
$\mathbb{E}[\iota(S(T))^4] < \infty$.
We define the process as $S(T) = x + W(N(T))$.  
Conditionally on $N(T)=y$, we have
\[
S(T)\mid N(T)=y \sim \mathcal{N}(x,y).
\]
Hence,
\[
\mathbb{E}[S(T)^4 \mid N(T)=y] = x^4 + 6x^2y + 3y^2 .
\]
By the tower  property,
\[
\mathbb{E}[|S(T)|^4]
= \mathbb{E}\big[\mathbb{E}[S(T)^4 \mid N(T)]\big]
= x^4 + 6x^2\mathbb{E}[N(T)] + 3\mathbb{E}[N(T)^2].
\]
If $N(T)\sim \text{Poisson}(\lambda T)$, then
\[
\mathbb{E}[|S(T)|^4]
= x^4 + 6x^2\lambda T + 3\big(\lambda T + (\lambda T)^2\big).
\]
Then \(  \mathbb{E}[|S(T)|^4] <\infty\).
Provided that the moments of the process $N(T)$ are finite, it follows that $\mathbb{E}[\iota(S(T))^4] < \infty$.
By the same calculus, we proved  $E[\kappa (V(T))^4] < \infty$.
Consequently:
\[ E[h(S(T), V(T)) ^2]  < \infty. \]
Thus, $h(S(T), V(T)) $ belongs to $L^2(\Omega)$.
Next, we verify that there exists an integrable random variable \(H\) such that
\[
\left|
\frac{\partial}{\partial x}
\big(\iota(S(T))\kappa(V(T))\big)
\right|
\le H.
\]
We have 
\[\iota(S(T))\,\kappa(V(T))
=
\max(K-S(T),0)\,\max(V(T)-L,0).
\]
Since only \(S(T)\) depends on the initial condition \(x\), we have
\[
\frac{\partial}{\partial x}
\big(\iota(S(T))\kappa(V(T))\big)
=
\kappa(V(T))
\frac{\partial}{\partial x}\iota(S(T)).
\]
Now,
\[
\iota(y)=\max(K-y,0),
\]
hence
\[
\iota'(y)
=
-\mathbf 1_{\{y<K\}}.
\]
Therefore,
\[
\frac{\partial}{\partial x}\iota(S(T))
=
-\mathbf 1_{\{S(T)<K\}}
\frac{\partial S(T)}{\partial x}.
\]
Since
\[
S(T)=x+W(N(T)),
\]
we obtain
\[
\frac{\partial S(T)}{\partial x}=1.
\]
Consequently,
\[
\left|
\frac{\partial}{\partial x}
\big(\iota(S(T))\kappa(V(T))\big)
\right|
=
\mathbf 1_{\{S(T)<K\}}
|\kappa(V(T))|.
\]
Since
\[
\kappa(V(T))
=
\max(V(T)-L,0),
\]
we have
\[
|\kappa(V(T))|
\le
L+|V(T)|.
\]
Hence,
\[
\left|
\frac{\partial}{\partial x}
\big(\iota(S(T))\kappa(V(T))\big)
\right|
\le
L+|V(T)|.
\]
Thus one may choose
\[
H=L+|V(T)|.
\]
It remains to verify that \(H\) is integrable. Since
\[
V(T)=v+W_2(N(T)),
\]
and \(W_2(N(T))\) follows a normal inverse Gaussian distribution, it admits finite moments of order one. Therefore,
\[
\mathbb E[|V(T)|]<\infty,
\]
which implies
\[
\mathbb E[H]<\infty.
\]
Hence, the required conditions are satisfied. Then, the delta corresponding  to the price is given by:
   \begin{equation*}
    \Delta_{ S} =\frac{1}{\sqrt{1 - \rho^2} }\mathbb{E}\Big[
    \iota(S(T))\kappa(V(T))\frac{W_1((N(T))-\tau}{N(T)}\Big],
\end{equation*}
where $\tau=W_1(0)$.
In the first method of computing $\Delta_S$, Monte Carlo simulations are used to estimate the delta using correlated Brownian motions and an inverse Gaussian jump process (IG). Specifically, $N(T)$ is simulated from an inverse Gaussian distribution with parameters $\nu = 1.0$ and $\mu_N = 1.0$, and $W_1, W_2$ are generated as Gaussian random variables with variance $N(T)$. The correlation between the Brownian components is set to $\rho = 0.5$, and the constant $v = 0.85$ appears in the delta kernel. The underlying   (NIG) process is given by:
\[
Y = c + \beta Z + \sqrt{Z} X,
\]
where $Z \sim IG(\delta, \gamma)$ and $X \sim N(0, 1)$. 
With  \(K=2.5\) and \(  L =1\),  the final estimated value obtained is:
\[
\Delta_{S}^{(1)} \approx 0.298849 \quad \text{(Monte Carlo Method 1)}.
\]

\subsubsection{Second Method: Numerical integration with NIG process}\label{numeric}
 \begin{equation}
    h(S(T), V(T)) =\iota(S((T)) \, \kappa(V(T)) = \max(K - S(T), 0) \times \max(V(T) - L, 0)
\end{equation}
and incorporating the threshold \(\epsilon = 0.1\), intensity \(\nu = 1.0\), and scale \(= 1.0\) via numerical integration of the integral \(I(\epsilon)=0.0316\), the estimator converges to the value:
\[
\Delta_{S}^{(2)}\approx 0.290339  
\]
The model accounts for the correlation \(\rho = 0.5\) between the Brownian motions \(B(T)\) and \(B_1(T)\), which influences the estimation of \(\Delta_{S}\). Specifically, the correlation is taken into account in the denominator of the estimator, which normalizes by \(\sqrt{1-\rho^2}\):
\[
\Delta_{S} = \frac{1}{\sqrt{1 - \rho^2}}\mathbb{E}\left[ \iota(S((T)) \, \kappa(V(T)) \, \frac{B_T - y}{T \sigma_\epsilon} \right].
\]
\begin{remark}
To assess the accuracy of the approximation based on the Lévy-Itô decomposition, we compare the resulting hedge with the one obtained from the exact variance--mean mixture representation of the NIG process. Using $(10^5)$ Monte Carlo simulations, the two methods yield
\[
\Delta_S^{(1)} = 0.298849,
\qquad
\Delta_S^{(2)} = 0.290339,
\]
where \(\Delta_S^{(1)}\) corresponds to the exact NIG simulation and \(\Delta_S^{(2)}\) to the approximation based on the Brownian replacement of small jumps.
The associated Monte Carlo standard errors are
\[
SE_1 = 0.003302,
\qquad
SE_2 = 0.003777,
\]
leading to the (95\%) confidence intervals
\[
[0.292377, 0.305320]
\]
and
\[
[0.282936,0.297741],
\]
respectively.
The relative difference between the two point estimates is approximately
\[
\frac{|\Delta_S^{(1)}-\Delta_S^{(2)}|}
{\Delta_S^{(1)}} \approx 2.8\%.
\]
Although the confidence intervals overlap, the overlap is limited and the difference between the point estimates exceeds the individual Monte Carlo standard errors. This indicates that the discrepancy cannot be attributed solely to simulation noise and is partly due to the truncation of small jumps and their replacement by a Brownian component.
Consequently, the exact variance--mean mixture representation should be regarded as the benchmark method whenever direct simulation of the inverse Gaussian random variable is available. Nevertheless, the approximation based on the Lévy-Itô decomposition provides a reasonably accurate estimate of the delta while avoiding the explicit simulation of infinitely many small jumps. The numerical results therefore suggest that the approximation captures the main contribution of the small-jump component, at the cost of a moderate bias of approximately (2.8\%) in the present experiment.

\end{remark}

\subsection{Ornstein Uhlenbeck process driving by NIG process}
The purpose of this example is to illustrate how the Malliavin representation provides an explicit expression for the delta of an Ornstein–Uhlenbeck process driven by a NIG process, without requiring explicit knowledge of its probability density function.

Consider the following SDE driven by a NIG process $L(t), t\in[0,T]$:
\begin{equation}\label{langevinjump}
   dS(t) = -\alpha S(t)dt + dL(t) = -\alpha S(t)dt + \sigma(\varepsilon) dB(t) + \int_{\left|z \right| > \varepsilon} z \tilde{N}(dt,dz). 
\end{equation}
Applying the It\^ o formula, we get that the solution of \eqref{langevinjump} is given by:
\begin{equation}\label{OUjump}
    S(t) = e^{-\alpha t}S_0 + \sigma(\varepsilon) \int_{0}^{t} e^{-\alpha (t-s)} dB(s) + \int_{0}^{t} \int_{\left|z \right| > \varepsilon} e^{-\alpha (t-s)}z \tilde{N}(ds,dz).
\end{equation}

Now we introduce the process $V(t)$   such that :\\
\begin{align}\label{scaledBM22}
 dV(t) &= -\alpha_1dY(t) +\sigma_V (\varepsilon)dB_1 (t)
+ \int_{\left| z \right| > \varepsilon} z \tilde{N}(dt, dz) \nonumber\\
&= -\alpha_1dY(t) + \sigma_V (\varepsilon)dB_1 (t)
+ \int_{\left| z \right| > \varepsilon} z    N (dt, dz) - \int_{\left| z \right| > \varepsilon}z l(dz)dt.
\end{align}
We assume     $ B(t)  $   and $ B_1(t) $ are correlated by the relation 
\begin{equation*} 
    dB_1(t)=\rho dB(t)+ \sqrt{1-\rho^2}  dB_2(t)
    \end{equation*}
with $B_2(t)$ is independent  of both  $ B(t)$   and $ B_1(t)$.
\begin{proposition}
Let \((S(t))_{t\in[0,T]}\) and \((V(t))_{t\in[0,T]}\) be the processes satisfying \eqref{OUjump} and \eqref{scaledBM22}, respectively. Assume that
\[
h(S(T),V(T))
=
\iota(S(T))\,\kappa(V(T))
\in L^2(\Omega),
\]
and that there exists an integrable random variable \(H\) such that
\[
\left|
\frac{\partial}{\partial x}
\big(\iota(S(T))\,\kappa(V(T))\big)
\right|
\le H.
\]
Then the delta corresponding to the process \((S(t))_{t\in[0,T]}\) is given by
\[
\Delta_S
=
 \frac{1}{\sqrt{1 - \rho^2} }\mathbb{E}\left[ h(S(T),V(T))\int_{0}^{T}\frac{e^{-\alpha t}}{T\sigma(\varepsilon)}dB(t)\right].
\]
where \(y=B(0)\).
\end{proposition}
    \begin{proof}
    From equation \eqref{langevinjump}, we get that the process $(S^c(t))_{t\in[0,T]}$ satisfies the following SDE:
\begin{equation}\label{f}
   \begin{cases}
dS^c(t) = -\alpha S(t)dt - \int_{\left|z\right|> \varepsilon} z \nu(dz)dt+\sigma (\varepsilon )dB(t)  \\ 
S^c(0)=x.
\end{cases} 
\end{equation}

In this case we have
\(
\sigma_c(S^c(t)) = \sigma (\varepsilon ). 
\)
The solution of \eqref{f} is given by 
\[
S^c(t)
=
x e^{-\alpha t}
-
\frac{1-e^{-\alpha t}}{\alpha}
\int_{|z|>\varepsilon} z\,\nu(dz)
+
\sigma(\varepsilon)
\int_0^t e^{-\alpha (t-s)}\,dB(s).
\]
 In this case \(Z_{1}(t)= \frac{\partial S^{c}(t)}{\partial x}=e^{-\alpha t}\) and we choose $\theta(t)=\frac{1}{T}$.
Applying Theorem \ref{Theorem4.6}, we get
\begin{equation}
    \Delta_{S }=  \frac{1}{\sqrt{1 - \rho^2} }\mathbb{E}\left[ h(S(T),V(T))\int_{0}^{T}\frac{e^{-\alpha t}}{T\sigma(\varepsilon)}dB(t)\right]. 
\end{equation}
    \end{proof}
The purpose of this example is to illustrate how the Malliavin representation yields an explicit formula for the delta of an Ornstein-Uhlenbeck process driven by a NIG Lévy process without requiring explicit knowledge of its density. Although the density could in principle be recovered from the characteristic function of the process, the resulting sensitivity computations are generally more involved. The Malliavin approach provides instead a direct representation of the hedge as an expectation that can be efficiently estimated by Monte Carlo methods.
The parameter \(\varepsilon\) plays a dual role in the approximation of the NIG process. On the one hand, decreasing \(\varepsilon\) improves the approximation of the small-jump component by a Brownian motion, thereby reducing the truncation error. On the other hand, the Malliavin weight appearing in the delta representation contains the factor
\[
\frac{1}{\sigma(\varepsilon)},
\qquad
\text{where}
\qquad
\sigma^2(\varepsilon)=
\int_{|z|\le \varepsilon} z^2 \ell(dz).
\]
Since \(\sigma(\varepsilon)\) becomes small as \(\varepsilon\) decreases, the variance of the corresponding Monte Carlo estimator may increase. Consequently, although a smaller value of \(\varepsilon\) generally reduces the approximation bias, it may also lead to a less stable numerical estimator.
Therefore, the choice of \(\varepsilon\) involves a classical bias--variance trade-off. Very small values of \(\varepsilon\) provide a more accurate approximation of the NIG process but may substantially increase the variance of the Malliavin estimator, whereas larger values of \(\varepsilon\) improve numerical stability at the expense of a larger truncation error.
In practice, \(\varepsilon\) is chosen as a compromise between these two effects. A common strategy is to compute the estimator for several values of \(\varepsilon\) and select a value for which the estimate stabilizes while maintaining acceptable confidence intervals. The numerical results presented in Section \ref{numeric} indicate that, for \(\varepsilon = 0.1\), the approximation based on the Lévy-Itô decomposition produces a delta estimate that remains close to the benchmark value obtained from the exact variance--mean mixture representation of the NIG process, while preserving a reasonable Monte Carlo variance.

\section{Summary}
 In this paper, we developed a stochastic framework for the computation of sensitivities of energy derivatives under a variety of continuous and jump-driven models. Two complementary approaches were investigated. The first approach is based on probability density functions and applies to models for which the density is explicitly available, such as Ornstein--Uhlenbeck and CARMA processes. The second approach relies on Malliavin calculus and provides representations of the Greeks directly from the dynamics of the underlying process, without requiring explicit knowledge of the corresponding density.
The Malliavin framework was applied to several classes of models, including Ornstein-Uhlenbeck processes, jump diffusion driven by a compound Poisson process, time changed Brownian motion by an inverse Gaussian process, and Ornstein-Uhlenbeck processes driven by a normal inverse Gaussian process. When both approaches are applicable, the resulting sensitivity representations are theoretically equivalent and provide alternative ways of computing the same Greek.
The numerical examples illustrate the practical implementation of the proposed formulas and show that the corresponding delta estimates are in close agreement. In particular, the examples demonstrate that the Malliavin approach remains applicable in situations where the transition density is unavailable or difficult to characterize explicitly. These results highlight the flexibility of Malliavin calculus for sensitivity analysis in models involving jumps and non-Gaussian dynamics that frequently arise in energy markets.

\appendix
\section{Appendix}In this appendix, we provide the proof of the example presented in Section \ref{sec} using the density method.
 \subsection{A correlated two-factor price and one-factor temperature model} \label{proof of  example 2 density method}
We consider the stochastic differential system  equations:  

Introduce Brownian motions $W_1$, $W_2$, $W_3$ with covariance structure
\[
\langle W_i, W_j \rangle_t = \rho_{ij} \, t, \quad 
\rho_{ii} = 1, \quad \rho_{ij} = \rho_{ji} \in [-1,1], \quad i,j=1,2,3.
\]
Consider the system of OU SDEs:
\begin{align}
dX_1(t) &= -\alpha_1 X_1(t)\, dt + b_1 \, dW_1(t), \\
dX_2(t) &= -\alpha_2 X_2(t)\, dt + b_2 \, dW_2(t), \\
dY(t)   &= -\alpha_3 Y(t)\, dt + b_3 \, dW_3(t),
\end{align}
with $\alpha_i > 0$ and $b_i \in \mathbb{R}$.
Let
\[
\mathbf{X} = 
\begin{pmatrix} X_1 \\ X_2 \\ Y \end{pmatrix}, \quad
\mu(t) = 
\begin{pmatrix} X_1(0)e^{-\alpha_1 t} \\ X_2(0)e^{-\alpha_2 t} \\ Y(0)e^{-\alpha_3 t} \end{pmatrix},
\]
and define the covariance matrix $\Sigma(t)$ with entries  
\[
\sigma_{ii} = \frac{b_i^2}{2\alpha_i} (1 - e^{-2\alpha_i t}), \quad
\sigma_{ij} = \frac{\rho_{ij} b_i b_j}{\alpha_i + \alpha_j} \left( 1 - e^{-(\alpha_i + \alpha_j)t} \right), \quad i \neq j.
\]
Then $(X_1(t), X_2(t), Y(t))$ is trivariate Gaussian distributed with density
\[
g_{X_1,X_2,Y}(x_1,x_2,y) = \frac{1}{(2\pi)^{3/2} \sqrt{\det \Sigma(t)}} 
\exp\Bigg[-\frac{1}{2} (\mathbf{x}-\mu(t))^\top \Sigma(t)^{-1} (\mathbf{x}-\mu(t))\Bigg].
\]
The ln-density is 
\[
\ln g_{X_1,X_2,Y}(x_1,x_2,y) = -\frac{3}{2} \ln(2\pi) - \frac{1}{2} \ln \det \Sigma(t) - \frac{1}{2} (\mathbf{x}-\mu)^\top \Sigma^{-1} (\mathbf{x}-\mu).
\]
Let $\Sigma^{-1} = (\eta_{ij})$. Then the derivative of $\ln g$ with respect to $x_1$ is
\[
\frac{\partial}{\partial x_1} \ln g_{X_1,X_2,Y}(x_1,x_2,y) 
= - \sum_{j=1}^3 \eta_{1j} (x_j - \mu_j)
= - \eta_{11} (x_1 - \mu_1) - \eta_{12} (x_2 - \mu_2) - \eta_{13} (y - \mu_3).
\]
Using the Schwartz model in \cite{benth2010computation, pricesevidence} which defined by
\begin{equation}
    S(t)=S(0) \exp(X_1(t)+X_2(t)),
\end{equation}
so by letting $n=2$ we get $S(T)=h_1(T,X_1(T),X_2(T))=S(0) \exp(X_1(T)+X_2(T))=\exp(\ln S(0)+X_1(T)+X_2(T))$.
For the temperature process $V(t)$,  
\begin{equation}
    V(t)=V(0)+Y(t).
\end{equation}
By letting $m=1$, we get $V(T)= h_2(Y(T) )=V(0)+ Y(T) $.
The payoff function can be represented as:
\begin{align}
    h(S(T), V(T))&=h(\exp(\ln S(0)+X_1(T)+X_2(T)),  V(0)+Y(T) )\nonumber\\
    &=G(T, \ln S(0)+X_1(T),X_2(T), V(0)+Y(T)).
\end{align}
 In this case $ \zeta_1(s)=\ln(s)$  and $\zeta_2(s) =s.$

In order  to use Proposition \ref{prop 2.1} , we should verify the assumptions in I). 
The derivative with respect to $x_1$ of $g$ is given by.
\[
\frac{\partial g}{\partial x_1} (x_1,x_2,y)
= - g(x_1,x_2,y)\,\Big[\, \eta_{11}(x_1 - \mu_1) + \eta_{12}(x_2 - \mu_2) + \eta_{13}(y - \mu_3) \,\Big].
\]
There exists a constant $C>0$ such that
\(
\Big|\frac{\partial g}{\partial x_1}(x_1,x_2,y)\Big|
\le  C(1 + |x_1| + |x_2| + |y|)\, g(x_1,x_2,y).
\)
Then,

\begin{align*}
    &\left| 
    G(T,x_1 + \zeta_{1}\big(S(0)),x_2,   y + \zeta_{2}\big(V(0)))\,
    \frac{\partial}{\partial x_1} g(x_1 - z,x_2,  y) 
    \right|\\&
    \leq C \times |G(T,x_1+ \zeta_{1}\big(S(0)),x_2,  y + \zeta_{2}\big(V(0)))|\,
     (1+ |x_1-z| + |x_2| + |y|)\, g(x_1-z,x_2,y) ).
    \end{align*}
    Using the Cauchy-Schwarz inequality 
    \begin{align*}
        &\int_{\mathbb{R}^2}  |G(T,x_1+ \zeta_{1}\big(S(0)),x_2,  y + \zeta_{2}\big(V(0)))|\,
     (1 + |x_1-z| +|x_2|+ |y|)\, g(x_1-z,x_2,y) dx_1dx_2dy \\&\leq (\int_{\mathbb{R}^2}  G^2(T,x_1+ \zeta_{1}\big(S(0)),x_2,  y + \zeta_{2}\big(V(0))g(x_1-z,x_2,y)dx_1dx_2dy)^\frac{1}{2}\\&\qquad\times(\int_{\mathbb{R}^2}  (1 + |x_1-z| + |x_2|+ |y|)^2  g(x_1-z,x_2,y)dx_1dx_2dy)^\frac{1}{2}\\& =\mathbb{E} \Big[ |G(T,X_1(T)+z +\zeta_{1}(S(0)),X_2(T),Y (T)+\zeta_{2}(V(0))|^2\Big]^\frac{1}{2}\\&\qquad\times(\int_{\mathbb{R}^2}  (1 + |x_1| + |x_2|+ |y|)^2  g(x_1,x_2,y)dx_1dx_2dy)^\frac{1}{2}< \infty.
    \end{align*}
    
 Therefore, according to Proposition \ref{prop 2.1}, the delta with respect to  $S(0)$  is given by:
\begin{align*}
\Delta_S
&= \frac{\partial C}{\partial S(0)} \\
&=e^{-rT}\mathbb{E} \Big[ G(T,X_1(T)+z +\zeta_{1}(S(0)),X_2(T),Y (T)+\zeta_{2}(V(0))\big) \\
&\quad \times \frac{\eta_{11} (X_1 - \mu_1) +\eta_{12} (X_2 - \mu_2) + \eta_{13} (Y - \mu_3)}{S(0)} \Big].\\
\end{align*}

\bibliographystyle{plain}
\bibliography{bib}
\end{document}